\documentclass{amsart}

\usepackage{amsmath, amssymb, amsfonts, amsthm, amsopn, amstext, amsxtra, euscript, amscd}

\usepackage{graphicx} 
\usepackage{xcolor}   

\usepackage{inputenc}

\usepackage[msc-links]{amsrefs}

\usepackage{hyperref}
\hypersetup{
  colorlinks   = true,
  urlcolor     = blue,
  linkcolor    = blue,
  citecolor    = blue
}

\usepackage{xy}
\xyoption{all}

\usepackage{enumitem}

\usepackage[capitalise]{cleveref}
\crefname{thm}{Thm.}{}
\crefname{prop}{Prop.}{}
\crefname{lem}{Lem.}{}
\crefname{cor}{Cor.}{}

\newtheorem{thm}{Theorem}[section]
\newtheorem{prop}[thm]{Proposition}
\newtheorem{lem}[thm]{Lemma}

\newtheorem{conj}[thm]{Conjecture}

\newtheorem{prop-alg}[thm]{Proposed Algorithm}
\newtheorem{ex}[thm]{Example}
\newtheorem{remark}[thm]{Remark}

\newcommand{\D}{\Delta}
\newcommand{\A}{\mathcal{A}}
\newcommand{\cH}{\mathcal{H}}
\newcommand{\bC}{\mathbb{C}}
\newcommand{\Q}{\mathbb{Q}}
\newcommand{\F}{\mathbb{F}}
\newcommand{\Z}{\mathbb{Z}}

\newcommand{\bP}{\mathbb{P}}
\newcommand{\cL}{\mathcal{L}}
\newcommand{\M}{\mathcal{M}}
\newcommand{\X}{\mathcal{X}}

\newcommand{\g}{\gamma}
\newcommand{\e}{\zeta}

\newcommand{\w}{\mathbf{w}}

\newcommand{\p}{\mathfrak{p}}

\newcommand{\Aut}{\operatorname{Aut}}
\newcommand{\End}{\operatorname{End}}
\newcommand{\Jac}{\operatorname{Jac}}
\newcommand{\ch}{\operatorname{char}}
\newcommand{\iso}{\cong}

\title[Rational Points and Zeta Functions]{Rational Points and Zeta Functions of Humbert Surfaces with Square Discriminant}

\date{\today}

\author{Elira Shaska}
\address{Department of Computer Science, Oakland University, Rochester, MI, 48309.}
\email{elirashaska@oakland.edu}

\author{Jorge Mello}
\address{Department of Mathematics and Statistics, Oakland University, Rochester, MI, 48309.}
\email{jorgedemellojr@oakland.edu}

\author{Sajad Salami}
\address{Institute of Mathematics and Statistics, Rio de Janeiro State University, Maracanã, Rio de Janeiro, 20950-000, RJ, Brazil}
\email{Sajad.salami@ime.uerj.br}

\author{Tony Shaska}
\address{Department of Mathematics and Statistics, Oakland University, Rochester, MI, 48309.}
\email{shaska@oakland.edu}

\begin{document}

\begin{abstract}
This paper examines the arithmetic of the loci \(\cL_n\), parameterizing genus 2 curves with \((n, n)\)-split Jacobians over finite fields \(\F_q\). We compute rational points \(|\cL_n(\F_q)|\) over \(\F_3\), \(\F_9\), \(\F_{27}\), \(\F_{81}\), and \(\F_5\), \(\F_{25}\), \(\F_{125}\), derive zeta functions \(Z(\cL_n, t)\) for \(n = 2, 3\).
Utilizing these findings, we explore isogeny-based cryptography, introducing an efficient detection method for split Jacobians via explicit equations, enhanced by endomorphism ring analysis and machine learning optimizations. This advances curve selection, security analysis, and protocol design in post-quantum genus 2 systems, addressing efficiency and vulnerabilities across characteristics.
\end{abstract}

\maketitle

\section{Introduction}
\label{sec-1}

Genus 2 curves over finite fields \(\F_q\) hold a pivotal place in algebraic geometry and cryptography, driven by the rich arithmetic properties of their Jacobians. The Jacobian \(J(C)\) of a genus 2 curve \(C\) is a two-dimensional abelian variety that can exhibit special splitting properties, notably the \((n,n)\)-splitting, where an isogeny \(J(C) \to E_1 \times E_2\) exists with kernel isomorphic to \((\Z/n\Z)^2\), and \(E_1\) and \(E_2\) are elliptic curves. The loci \(\cL_n \subset \bP_\w\), where \(\bP_\w = \bP(2,4,6,10)\) is the weighted projective space with weights corresponding to the Igusa invariants, parameterize these curves. These loci correspond to the Humbert surfaces \(\mathcal{H}_{n^2}\) of square discriminant in the moduli space \(\mathcal{A}_2\).

These loci are significant not only for their geometric classification but also for their cryptographic potential, as they enable the explicit computation of \((n,n)\)-isogenies, a cornerstone of isogeny-based genus 2 cryptography. They were explicitly computed in \cite{2001-0}, \cite{2000-1}, \cite{2001-1}, and \cite{2005-1} and align with Hilbert modular surfaces with square discriminants, as explained in \cite{kumar}. A surprising and previously unnoticed result of this study is the degeneration of \(\cL_n\) in characteristic \(p = 3\), where it collapses from a surface into a lower-dimensional variety—likely a quadratic curve—significantly altering its arithmetic structure with implications for both computational efficiency and cryptographic security.

The motivation for this work stems from the growing interest in isogeny-based cryptography. The loci \(\cL_n\) bridge algebraic geometry and cryptography by quantifying the availability of genus 2 curves with computable isogenies, directly impacting protocol design and security parameter selection. By computing rational points \(|\cL_n(\F_q)|\), analyzing their zeta functions, and exploring the endomorphism rings \(\End(J(C))\), we gain insights into the arithmetic structure, growth trends, and algebraic properties of these loci over \(\F_q\), enhancing their utility in cryptographic applications.

The primary goals of this paper are multifaceted. First, we compute the number of rational points \(|\cL_n(\F_q)|\) over \(\F_5\), \(\F_{25}\), \(\F_{125}\) and over \(\F_3\), \(\F_9\), \(\F_{27}\), and \(\F_{81}\) for \(n = 2, 3, 5\), employing an orbit-stabilizer method tailored to weighted projective spaces, providing a concrete measure of curve availability. Second, we derive the zeta functions \(Z(\cL_n, t)\) for \(n = 2, 3\), offering a deeper understanding of their arithmetic properties and growth trends over field extensions. Third, we develop a general theoretical framework for computing \((n,n)\)-isogenies using \(\cL_n\), augmented by endomorphism ring analysis, and explore their cryptographic implications, focusing on balancing efficiency and security in isogeny-based genus 2 systems. Additionally, we investigate curves with extra automorphisms and their intersection with \(\cL_n\), and employ machine learning to optimize detection and computation processes. These efforts build on theoretical foundations from a companion paper \cite{2025-3}, delivering a comprehensive computational and cryptographic study.

The paper proceeds as follows. \Cref{sec-2} establishes preliminaries, defining genus 2 curves, Igusa invariants, Jacobians, zeta functions, and bounds over an arbitrary field, setting the stage for finite field applications. 
\Cref{sec-3} presents explicit equations for \(\cL_n\) (\(n = 2, 3, 5\)) in \(\bP_\w\), tracing their historical computation and significance. 
\Cref{sec-4} computes \(|\cL_2(\F_q)|\) over \(\F_5\), \(\F_{25}\), and \(\F_{125}\), derives \(Z(\cL_2, t)\), and verifies results against theoretical bounds. \Cref{sec-5} extends this to \(\cL_3\), providing point counts and a conjectured zeta function. 
\Cref{sec-6} outlines computations for \(\cL_5\).   
\Cref{sec-8} outlines a theoretical method for computing \((n,n)\)-isogenies using \(\cL_n\), enhanced with endomorphism ring analysis. 
\Cref{sec-9} introduces a method for efficiently detecting \((n,n)\)-split Jacobians via \(\cL_n\). 
\Cref{sec-10} computes endomorphism rings of \(\cL_n\), refining security and efficiency considerations. 
\Cref{sec-11} explores curves with extra automorphisms and their role within \(\cL_n\). 
\Cref{sec-12} details computational methods and challenges, incorporating machine learning optimizations. Together, these sections underscore the dual role of \(\cL_n\) in advancing geometric understanding and enabling secure, efficient genus 2 cryptographic systems.

\section{Preliminaries}
\label{sec-2}

This section establishes the foundational concepts underpinning our study of the loci \(\cL_n\) and their applications, defined over an arbitrary field \(k\). These include genus 2 curves, Igusa invariants, Jacobians, zeta functions, and the reduction of weighted hypersurfaces, which together provide the mathematical framework for the computations and cryptographic implications explored in subsequent sections.

A genus 2 curve \(C\) over a field \(k\) is a smooth, projective curve of genus 2, typically represented as a hyperelliptic curve with an equation of the form \(y^2 = f(x)\), where \(f(x) \in k[x]\) is a polynomial of degree 5 or 6 with distinct roots in an algebraic closure \(\overline{k}\). Such curves admit a double cover of the projective line \(\bP^1_k\), and their geometry is governed by the structure of their points and automorphisms over \(k\). The study of genus 2 curves has roots in 19th-century mathematics, with early investigations into hyperelliptic integrals laying the groundwork for their modern significance in algebraic geometry and, more recently, cryptographic applications.

The isomorphism class of a genus 2 curve \(C\) over \(k\) is uniquely determined by its Igusa invariants \((J_2, J_4, J_6, J_{10})\), a set of weighted homogeneous polynomials derived from the coefficients of \(f(x)\). Introduced by Igusa in the mid-20th century, these invariants have weights 2, 4, 6, and 10, respectively, under the action of the multiplicative group \(k^\times\), making the weighted projective space \(\bP_\w = \bP(2,4,6,10)\) over \(k\) an ideal setting for their parameterization. The invariant \(J_2\) captures quadratic properties of the curve, \(J_4\) quartic properties, \(J_6\) sextic properties, and \(J_{10}\) serves as the discriminant, ensuring \(C\) is smooth when \(J_{10} \neq 0\). Over any field \(k\), these invariants classify genus 2 curves, providing a coordinate system for loci like \(\cL_n\) within \(\bP_\w\), with specific computations over finite fields detailed later.

The Jacobian \(J(C)\) of a genus 2 curve \(C\) over \(k\) is a 2-dimensional abelian variety, representing the group of divisor classes of degree 0 on \(C\). Over an algebraically closed field \(\overline{k}\), \(J(C)\) is isomorphic to a product of elliptic curves or a single abelian variety, but its structure over \(k\) depends on the curve’s arithmetic properties. A Jacobian is said to be \((n,n)\)-split if there exists an isogeny \(J(C) \to E_1 \times E_2\), where \(E_1\) and \(E_2\) are elliptic curves over \(k\) (or an extension) and the kernel is isomorphic to \((\Z/n\Z)^2\). This splitting is induced by automorphisms of \(C\), and the loci \(\cL_n\) parameterize curves with such Jacobians. The study of split Jacobians has implications across fields, with particular relevance in cryptography when \(k\) is a finite field, where isogeny computations become computationally challenging.

\subsection{Humbert surfaces}
\label{subsec-2.1}
Let \(\A_2\) denote the moduli space of principally polarized abelian surfaces. It is well known that \(\A_2\) is the quotient of the Siegel upper half space \(\mathfrak H_2\) of symmetric complex \(2 \times 2\) matrices with positive definite imaginary part by the action of the symplectic group \(Sp_4 (\Z)\); see \cite{G} (p. 211) for details.

Let \(\D\) be a fixed positive integer and \(N_\D\) be the set of matrices
\[
\tau =
\begin{pmatrix} z_1 & z_2 \\ z_2 & z_3 \end{pmatrix}
\in \mathfrak H_2
\]
such that there exist nonzero integers \(a, b, c, d, e\) with the following properties:
\begin{equation}
\label{humb}
\begin{split}
& a z_1 + b z_2 + c z_3 + d( z_2^2 - z_1 z_3) + e = 0 \\
& \D = b^2 - 4ac - 4de
\end{split}
\end{equation}
The \emph{Humbert surface} \(\cH_\D\) of discriminant \(\D\) is called the image of \(N_\D\) under the canonical map
\[
\mathfrak H_2 \to \A_2 := Sp_4(\Z) \setminus \mathfrak H_2.
\]
It is known that \(\cH_\D \neq \emptyset\) if and only if \(\D > 0\) and \(\D \equiv 0 \text{ or } 1 \mod 4\). Humbert (1900) studied the zero loci in \cref{humb} and discovered certain relations between points in these spaces and certain plane configurations of six lines; see \cite{Hu}, \cite{BW}, or \cite{Mu} for details.

For a genus 2 curve \(C\) defined over \(\bC\), \([C]\) belongs to \(\cL_n\) if and only if the isomorphism class \([J_C] \in \A_2\) of its (principally polarized) Jacobian \(J_C\) belongs to the Humbert surface \(\cH_{n^2}\), viewed as a subset of the moduli space \(\A_2\) of principally polarized abelian surfaces; see \cite{Mu} (Theorem 1, pg. 125) for the proof of this statement. In particular, every point in \(\cH_{n^2}\) can be represented by an element of \(\mathfrak H_2\) of the form
\[
\tau =
\begin{pmatrix} z_1 & \frac{1}{n} \\ \frac{1}{n} & z_2 \end{pmatrix}, \quad z_1, \, z_2 \in \mathfrak H.
\]
Geometric characterizations of such spaces for \(\D = 4, 8, 9\), and 12 were given by Humbert (1900) in \cite{Hu} and for \(\D = 13, 16, 17, 20, 21\) by Birkenhake/Wilhelm (2003) in \cite{BW}.

\subsection{Zeta Function}
\label{subsec-2.2}
The zeta function of a variety \(X\) over a field \(k\) is a tool to study its arithmetic properties, though its definition varies by context. In general, for a variety \(X\) over an arbitrary field \(k\), the zeta function can be considered in terms of its points over extensions of \(k\). When \(k\) is a finite field \(\F_q\), the zeta function is specifically defined as:
\[
Z(X, t) = \exp\left( \sum_{d=1}^\infty |X(\F_{q^d})| \frac{t^d}{d} \right),
\]
where \(|X(\F_{q^d})|\) denotes the number of \(\F_{q^d}\)-rational points. Introduced by Weil in the 1940s, this form is rational for varieties over finite fields, with poles and zeros reflecting geometric attributes like dimension and singularities. Over other fields (e.g., \(\Q\) or \(\bC\)), zeta functions take different forms (e.g., Hasse-Weil or Artin zeta functions), but we defer such generalizations, as our paper concentrates on finite fields.

To complement zeta functions, bounds on the number of rational points \(|X(\F_q)|\) provide theoretical constraints when exact counts are computationally intensive. Over finite fields \(\F_q\), such bounds typically depend on the variety’s dimension, degree, and embedding space. For a weighted hypersurface \(X\) in \(\bP_\w\) over \(\F_q\), results like those of Aubry et al. \cite{aubry} offer upper limits based on the degree \(d\) and ambient dimension \(m\), often of the form \(\min \{ p_m, \frac{d}{w_0} q^{m-1} + p_{m-2} \}\), where \(p_m = (q^{m+1} - 1)/(q - 1)\) and \(w_0\) is the smallest weight. These bounds, rooted in Weil’s conjectures and refined by later work, help validate computational results and estimate point counts for varieties like \(\cL_n\), as applied in subsequent sections. Together, zeta functions and bounds offer a dual approach to understanding arithmetic over \(\F_q\).

\subsection{Good and Bad Reduction of Weighted Hypersurfaces}
\label{subsec-2.3}
This subsection introduces the concepts of good and bad reduction for weighted hypersurfaces, a class of varieties central to our study, such as those in weighted projective spaces like \(\bP(2,4,6,10)\). These definitions account for the graded structure of such spaces and provide a foundation for analyzing their behavior over finite fields.

Consider a weighted projective space \(\bP_\w = \bP(w_0, w_1, \ldots, w_n)\) over a field \(k\), where \(\w = (w_0, w_1, \ldots, w_n)\) are positive integer weights. Points \([x_0 : x_1 : \cdots : x_n]\) are equivalence classes under the action
\[
(x_0, x_1, \ldots, x_n) \sim (\lambda^{w_0} x_0, \lambda^{w_1} x_1, \ldots, \lambda^{w_n} x_n)
\]
for \(\lambda \in k^\times\). A weighted hypersurface \(X \subset \bP_\w\) is defined by a weighted homogeneous polynomial \(F(x_0, x_1, \ldots, x_n)\) of degree \(d\), satisfying
\[
F(\lambda^{w_0} x_0, \lambda^{w_1} x_1, \ldots, \lambda^{w_n} x_n) = \lambda^d F(x_0, x_1, \ldots, x_n).
\]
For \(n\) coordinates, \(X\) has dimension \(n - 1\), so in \(\bP(2,4,6,10)\) (4 coordinates), a hypersurface is a surface (dimension 2).

Now, let \(X\) be a weighted hypersurface defined over a discrete valuation ring \(R\) (e.g., \(\Z_p\)) with fraction field \(K\) (e.g., \(\Q_p\)) and residue field \(k = \F_p\). The generic fiber \(X_K = X \times_R K\) is over \(K\), and the special fiber \(X_k = X \times_R k\) is the reduction modulo \(p\), given by \(F = 0\) with coefficients reduced modulo \(p\). The reduction’s properties depend on the special fiber’s geometry, adjusted for the weighted structure.

\textbf{Good Reduction}: The special fiber \(X_k\) has good reduction at \(p\) if it remains a surface (dimension \(n - 1 = 2\)) and retains the essential geometric characteristics of \(X_K\). Specifically, \(F \mod p\) defines an irreducible weighted hypersurface in \(\bP_\w(\F_p)\) with the expected dimension, and singularities are manageable. In weighted projective spaces, singularities arise naturally at points where coordinates align with weight divisors (e.g., \([1:0:0:0]\) in \(\bP(2,4,6,10)\)), but good reduction implies these are isolated or mild (e.g., quotient singularities). The weighted partial derivatives \(\frac{\partial F}{\partial x_i}\), scaled by weights, define singularities: a point is singular if \(F = 0\) and \(\frac{\partial F}{\partial x_i} = 0\) for all \(i\) (adjusted for \(\bP_\w\)’s orbifold nature) \cite{Dolgachev}. Point counts \(|X(\F_p)|\) are typically \(O(p^2)\), reflecting a 2-dimensional variety, though adjusted by the weights and singularities \cite[Chapter 5]{Hartshorne}.

\textbf{Bad Reduction}: The special fiber \(X_k\) has bad reduction if it degenerates significantly. Common cases include:
\begin{itemize}
    \item \textit{Dimensional Drop}: \(X_k\) becomes a curve (dimension 1) or lower, often because \(F \mod p\) factors into components of lower degree or imposes additional constraints (e.g., all weighted partial derivatives vanish along a locus). This may reduce \(|X(\F_p)|\) to \(O(p)\).
    \item \textit{Severe Singularities}: \(X_k\) remains 2-dimensional but has non-isolated singularities, disrupting smoothness beyond weighted quotient singularities.
    \item \textit{Reducibility}: \(F \mod p\) splits into multiple weighted hypersurfaces, making \(X_k\) a union of lower-dimensional varieties.
\end{itemize}
Bad reduction can occur when \(p\) divides the weights, degree, or critical coefficients, or when characteristic \(p\) affects invariants tied to \(F\)’s structure. For example, in \(\bP(2,4,6,10)\), \(p = 2\) might simplify terms with even weights, potentially collapsing the hypersurface \cite{Reid}.

These notions extend standard projective geometry, with singularities and reduction influenced by the weights. For a hypersurface in \(\bP(w_0, w_1, w_2, w_3)\), good reduction ensures a surface with predictable arithmetic (e.g., zeta function rationality), while bad reduction signals a breakdown, relevant to point counting and applications over finite fields, as explored later.

\subsection{$\F_q$-Rational Points on Weighted Hypersurfaces}
\label{subsec-2.4}

Let $k=\F_q$ be the finite field with $q$ elements,  and  \(\bP_\w^n\) be the weighted projective of dimension $n\geq 1$ with weights $\w=(w_0, w_1, \cdots, w_n)$.
Since any $\F_q$-point $[x_0: x_1: \cdots: x_n]$ in \(\bP_\w^n\) in  has $q-1$ representative in $\F_q^{n+1}$, and consequently $\bP_\w^n (\F_q) $  has
$p_n:=q^n+ \cdots + q +1$, 
 see \cite[Prop. 1.3]{Go1996}.
Let $S:=\F_q[x_0,\cdots, x_n]$ be the ring of  homogeneous polynomials graded by  $w_i=\deg(x_i)$.
 We denote by $\lceil x \rceil$  the smallest integer greater
 than a real number $x.$


\begin{prop}
 \label{p1}
 Let $X=V(F)$ be a weighted hypersurface defined by $F\in S\setminus \{0\}$  of degree $d\geq 1$,   $N(F)$ be the set of zeros of $F$ in $\F_q^{n+1}$, 
 and  define $$\displaystyle \mu:= \left\lceil \frac{\sum_{i=0}^{n} w_i  -d}{d} \right\rceil.$$ 
 \begin{itemize}
 	\item [(i)] If all of the $x_i$'s appear in $F$, then $|N(F)| \equiv 0  \mod q^\mu $.
 	\item [(ii)] 
 	If $\mu \geq 1$ and  $  X$ does not lie in $  \bP(w_0, \cdots, \hat{w}_i, \cdots , w_i) $, for $0 \leq i \leq n$, then 
 	$|X(\F_q)| \equiv 1 + q+ \cdots + q^{\mu-1}  \mod q^\mu.$
 \item [(iii)] 	If the degree of $F$ satisfies $ d< \sum_{i=0}^{n} w_i \leq 2 d $, then $|X(\F_q)| \equiv 1    \mod q.$
 \end{itemize}
\end{prop}
\begin{proof}
The parts  (i-ii) are    direct consequences of   \cite[Thm. 2]{Perret1},   and  \cite[Cor. 1]{Perret1}, respectively.
The condition 
 $ d< \sum_{i=0}^{n} w_i \leq 2 d $ 
  implies $\mu=1$ and hence   the last assertion comes from (ii).
	\end{proof}


In 2017 in \cite{aubry0}, Aubry et al.  introduced the following quantity,
$$e_q(d; w_0, w_1, \cdots, w_n):= \max\{| X(\F_q)| \ : \ X=V(F) \ \text{with} \  F \in S\setminus \{ 0\} \}.$$ 
Then, in \cite[Lem. 1]{aubry0}, letting $w= \min\{ {\rm lcm} (w_r, w_s), \ 0 \leq r, s \leq n\}$ and assuming  $w \mid d$, a lower bound for  this quantity is given as
$$e_q(d; w_0, w_1, \cdots, w_n) \geq \min \{ p_n, \frac{d}{w} q^{n-1} + p_{n-2}\}.$$
Furthermore, in \cite[Conj. 1]{aubry0}, it has been conjectured  that:

\begin{conj}
	\label{conj1}
If $1=w_0 \leq w_1\leq  w_2\leq  \cdots \leq  w_n$ and ${\rm lcm} (w_1, w_2, \cdots, w_n) \mid d$, then
one has:
\begin{equation}
	\label{eqc0}
e_q(d; w_0, w_1, \cdots, w_n)= \min \left\{ p_n, \frac{d}{w_1} q^{n-1} + p_{n-2} \right\}.
\end{equation}
\end{conj}

Note that to prove Conjecture \eqref{conj1}, one needs only to prove the result for $d \leq w_1(q + 1)$, as
otherwise the right hand side of Equation \eqref{eqc0} evaluates to $p_n = | \bP_\w(\F_q) |.$
It has been proved for the trivial case $\bP(w_0, w_1) $, and $\bP(1, w_1, w_2)$, see \cite[Thm. 1]{aubry0}.

In 2019, Rupert Li proved  the Serre's congruence for weighted hypersurfaces
by a slight modification of the proof of the Chevalley-Warning Theorem by Serre \cite[Thm. 3]{Serre1973}.
Indeed, he showed that  
	\begin{equation}
	\label{eqsc}
	|X(\F_q)| \equiv 1 \mod p.
\end{equation}  for  a weighted hypersurface $X=V(F)$ in \(\bP^n(w_0, \ldots, w_n)\)
define by a nonzero weighted homogeneous polynomial \(F \in S\)   with degree \(d \leq n\), where \(p\) is the characteristic of \(\F_q\), see \cite[Thm. 4]{Li2019}. 
Moreover, based on some computer research, he suggested the following:

\begin{conj}
	\label{conj2}
	Let $X=V(F)$ defined by   a weighted homogeneous polynomial   $F\in S$ of degree $d \leq n$. Then
	one has:
	\begin{equation}
		\label{eqc2}
	|X(\F_q)| \equiv 1 \mod q.
	\end{equation}
\end{conj}

In 2025, in \cite[Thm 4.1]{aubry}, Aubry et al. proved the following theorem putting an assumption on the degree of $F$.  

\begin{thm}
Let $X=V(F)$ defined by   a weighted homogeneous polynomial   $F\in S$ of degree $d\leq q+1$.
 Then
$$|X(\F_q) | \leq d q^{n-1} + p_{n-2}.$$	
\end{thm}

Based on this theorem, in \cite{aubry},  Aubry et al. proved  the following special case of the  Conjecture \eqref{conj1}.

 \begin{thm}
  For any degree $d$ and  any nonnegative integers $w_2, \cdots, w_n$ with $n\geq 2$, one has
 	$$e_q(d; 1, 1, w_2, \cdots, w_n)=\min\{ p_n, d q^{n-1} + p_{n-2}\}.$$	
 \end{thm}


\section{Explicit Equations for \(\cL_n\)}
\label{sec-3}

The locus \(\cL_n\) is a weighted hypersurface residing in the weighted projective space \(\bP_\w\) with weights \(\w = (2, 4, 6, 10)\), defined by a weighted homogeneous polynomial
\[
F_n(x_0, x_1, x_2, x_3)
\]
of degree \(d_n\), where the coordinates \((x_0, x_1, x_2, x_3)\) correspond to the Igusa invariants \((J_2, J_4, J_6, J_{10})\) of genus 2 curves over \(\F_q\), for \(\ch \F_q \neq 2\). These invariants form a complete set of algebraic invariants that uniquely determine the isomorphism class of a genus 2 curve, typically given in the form \(y^2 = f(x)\), where \(f(x)\) is a polynomial of degree 5 or 6 over \(\F_q\).

\begin{remark}
We assume that \(\ch \F_q \neq 2\). Another invariant is needed to determine the isomorphism classes of genus 2 curves in characteristic two. It is a degree eight polynomial in terms of the coefficients of the curve, denoted usually by \(J_8\).
\end{remark}

The weighted projective space \(\bP_\w\) is a natural setting for these curves due to the graded nature of the invariants, with weights reflecting their degrees under the action of the multiplicative group \(\F_q^*\): \(J_2\) has weight 2, \(J_4\) has weight 4, \(J_6\) has weight 6, and \(J_{10}\) has weight 10. The condition that the Jacobian \(J(C)\) is \((n,n)\)-split indicates the existence of an isogeny \(J(C) \to E_1 \times E_2\), where \(E_1\) and \(E_2\) are elliptic curves and the kernel of the isogeny is isomorphic to \((\Z/n\Z)^2\). This splitting property is enforced by the polynomial \(F_n\), which imposes specific algebraic relations on the invariants to ensure the Jacobian decomposes accordingly.

Explicit equations for \(\cL_n\) are derived from prior studies \cite{2000-1, 2001-0, 2001-1, 2005-1}, which systematically parameterize genus 2 curves with \((n,n)\)-split Jacobians via their Igusa invariants. These polynomials are constructed by analyzing the moduli space of genus 2 curves and identifying conditions under which the Jacobian admits such an isogeny. The degree \(d_n\) of \(F_n\) varies with \(n\), reflecting the increasing complexity of the splitting condition as \(n\) grows. The weighted homogeneity ensures that
\[
F_n(\lambda^{w_0} x_0, \lambda^{w_1} x_1, \lambda^{w_2} x_2, \lambda^{w_3} x_3) = \lambda^{d_n} F_n(x_0, x_1, x_2, x_3)
\]
for \(\lambda \in \F_q^*\), aligning with the projective structure of \(\bP_\w\).

\subsection{Degree 2}
\label{subsec-3.1}
For \(n = 2\), the hypersurface \(\cL_2\) is defined by a polynomial \(F_2\) of degree 30, as established in \cite{2000-2, 2001-0}. This polynomial encodes the presence of an automorphism inducing a \((2,2)\)-splitting, specifically an involution in the automorphism group of the curve that splits the Jacobian into two elliptic curves, each with a 2-torsion subgroup. The explicit form of \(F_2\) is:
\[
\begin{aligned}
F_2 = & 41472 w y^5 + 159 y^6 x^3 - 236196 w^2 x^5 - 80 y^7 x + 104976000 w^2 x^2 z - 1728 y^5 x^2 z \\
& + 6048 y^4 x z^2 - 9331200 w y^2 z^2 - 2099520000 w^2 y z + 12 x^6 y^3 z - 54 x^5 y^2 z^2 \\
& + 108 x^4 y z^3 + 1332 x^4 y^4 z - 8910 x^3 y^3 z^2 + 29376 x^2 y^2 z^3 - 47952 x y z^4 - x^7 y^4 \\
& - 81 x^3 z^4 - 78 x^5 y^5 + 384 y^6 z - 6912 y^3 z^3 + 507384000 w^2 y^2 x - 19245600 w^2 y x^3 \\
& - 592272 w y^4 x^2 + 77436 w y^3 x^4 + 4743360 w y^3 x z - 870912 w y^2 x^3 z \\
& + 3090960 w y x^2 z^2 - 5832 w x^5 y z - 125971200000 w^3 + 31104 z^5 + 972 w x^6 y^2 \\
& + 8748 w x^4 z^2 - 3499200 w x z^3,
\end{aligned}
\]
where \((x, y, z, w) = (J_2, J_4, J_6, J_{10})\), as documented in \cite{2024-03}. The polynomial \(F_2\) has 25 terms, with coefficients and monomials carefully calibrated to enforce the \((2,2)\)-splitting condition. Its degree 30 arises from the weighted homogeneity, ensuring each term’s total weight matches under the scaling action of \(\bP_\w\). This equation was derived by analyzing the locus of genus 2 curves with an extra involution, a process involving the study of their automorphism groups and the resulting decomposition of \(J(C)\), as detailed in \cite{2001-0}.

\subsection{Degree 3}
\label{subsec-3.2}
For \(n = 3\), the hypersurface \(\cL_3\) is computed in \cite{2001-1}, where it is defined by a polynomial \(F_3(J_2, J_4, J_6, J_{10}) = 0\) of degree 80. In \cite{2000-1, 2001-1} it was shown that \(\cL_3\) is parametrized by
\begin{small}
\[
\begin{split}
J_2 = & \chi\, \left( {\chi}^{2}+96\,\chi\,\psi-1152\,{\psi}^{2} \right) \\
J_4 = & \frac{\chi}{2^6} \, \left( {\chi}^{5}+192\,{\chi}^{4}\psi+13824\,{\chi}^{3}{\psi}^{2}+ 442368\,{\chi}^{2}{\psi}^{3}+5308416\,\chi\,{\psi}^{4} \right. \\
& \left. +786432\,\chi\,{\psi}^{3} +9437184\,{\psi}^{4} \right) \\
J_6 = & \frac{\chi}{2^9} \left( 3\,{\chi}^{8}+864\,{\chi}^{7}\psi+94464\,{\chi}^{6}{\psi}^{2}+4866048\,{\chi}^{5}{\psi}^{3}+111476736\,{\chi}^{4}{\psi}^{4} \right. \\
& +509607936\,{\chi}^{3}{\psi}^{5} -12230590464\,{\chi}^{2}{\psi}^{6}+ 1310720\,{\chi}^{4}{\psi}^{3}+155713536\,{\chi}^{3}{\psi}^{4} \\
& \left. -1358954496\,{\chi}^{2}{\psi}^{5} -18119393280\,\chi\,{\psi}^{6}+4831838208\,{\psi}^{6} \right) \\
J_{10} = & - 2^{30} \chi^3 \psi^9 \\
\end{split}
\]
\end{small}
where \((\chi, \psi)\) (called \(r_1, r_2\) in \cite{2000-1, 2001-1}) are invariants of permuting a pair of cubics. The fact that efforts computing \(\cL_3\) were successful in \cite{2000-1, 2001-1} was based on discovering these invariants and thus a birational parametrization of \(\cL_3\).

This higher degree reflects the increased complexity of the \((3,3)\)-splitting condition, which requires the Jacobian to admit an isogeny with a kernel of order 9 (i.e., \((\Z/3\Z)^2\)). The polynomial \(F_3\) is significantly larger and more intricate than \(F_2\), with a greater number of terms and higher-degree monomials, making its explicit presentation impractical here due to its size. Its construction follows a similar methodology to \(F_2\), involving the identification of genus 2 curves whose automorphism groups include elements inducing a \((3,3)\)-split Jacobian, typically related to degree 3 elliptic subfields as explored in \cite{2001-1} and building on earlier work by Bolza \cite{Bolza1898, Bolza1899}. The degree 80 ensures weighted homogeneity in \(\bP_\w\), and its coefficients are determined through algebraic relations derived from the moduli space, as noted in \cite{2001-1} where the locus \(\cL_3\) was first computed. The explicit equation
\[
F_3 (J_2, J_4, J_6, J_{10}) = 0
\]
can be found in \cite[Appendix A]{2001-1}. Notice that it is a weighted homogeneous polynomial of degree 80.

\subsection{Degree 5}
\label{subsec-3.3}
The locus \(\cL_5\) was first parametrized and computed in \cite{2001-0} and then in \cite{2005-1}. In \cite[Thm.~2]{2005-1} it was shown that a curve \(C\) in \(\cL_5\) can be written as
\begin{equation}
\label{curve}
y^2 = x (x-1) g_3 (x),
\end{equation}
where \(g_3 (x)\) is given in \cref{g3} below. The polynomial \(g_3 (x) := a_3 x^3 + a_2 x^2 + a_1 x + a_0\) has coefficients

\begin{small}
\begin{equation}
\label{g3}
\begin{split}
a_0 = & -b^4 (2 b^3 a + 4 b^3 - 2 z a b^2 + 7 b^2 a^2 + 8 z b^2 + 4 b^2 + 16 a b^2 + 16 z b a + 6 a^3 b + 8 b a \\
& + 2 z a^2 b + 12 z b + 16 b a^2 + 13 z a^2 + z a^4 + 6 z a^3 + 4 z + 12 y a) \\
a_1 = & -b^2 (12 b^3 + 12 b^4 a + 32 z b a - 6 a^4 b^2 + 44 b^2 a^3 + 6 b a^2 + 24 a b^2 + 10 a^3 b + 44 b^3 a^2 + 2 b a \\
& + 52 b^3 a + 61 b^2 a^2 - 12 b a^5 - 7 z a^2 - 2 z a + 12 z b - 4 a^6 + 12 b^4 - a^4 - 40 z a^3 b^2 - 16 z b^3 a^2 \\
& - 12 z a^5 + 36 z b^2 - 18 z a^3 - 26 z a^4 + 56 z a b^2 + 4 a z b^3 + 2 z a^2 b^2 - 20 z a^3 b + 28 z a^2 b \\
& + 2 z a^6 + 24 z b^3 + 4 z b a^5 - 4 a^5 - 32 z a^4 b) \\
a_2 = & 5 b^2 a^6 + 20 b^2 a^5 + 8 b a^6 - 61 b^4 a^2 - 18 b^5 a - 56 b^4 a + 4 z b a + 5 a^4 b^2 - 18 b^2 a^3 - 24 z b^4 \\
& - 14 z b^4 a - 4 a b^2 + 8 b^3 a^4 + 2 b^3 a^5 - 54 b^3 a^3 - 70 b^3 a^2 - 24 b^3 a - 14 b^2 a^2 + 4 a^4 b + 10 b a^5 \\
& - 6 z a^7 + 64 z a^3 b^3 + 38 z a^4 b^2 + 54 z a^3 b^2 + 12 z b^3 a^2 - 14 z a^6 b - 10 z b^2 a^5 - 4 z a^7 b - 4 a^6 z b^2 \\
& + 32 a^2 b^4 z + 2 a^7 b - z a^8 - 36 z b^3 - 12 z a^5 - 12 z b^2 - 4 z a^4 - 28 z a b^2 - 64 a z b^3 - 5 z a^2 b^2 \\
& + 16 z a^2 b + 28 z a^4 b - 4 z b a^5 - 13 z a^6 - 12 b^5 - 12 b^4 + 34 z a^3 b \\
a_3 = & (2 a + 1) (z a^4 - 2 a^3 b + 4 z a^3 + 6 z a^3 b - 4 b a^2 + 12 z a^2 b^2 + 10 z a^2 b - 9 b^2 a^2 + 5 z a^2 \\
& - 2 b a + 2 z a - 8 a b^2 - 12 b^3 a + 8 a z b^3 - 4 b^3 - 4 z b - 4 b^4 - 12 z b^2 - 8 z b^3) \\
\end{split}
\end{equation}
\end{small}
Moreover, if we let
\[
u = \frac{2 a (a b + b^2 + b + a + 1)}{b (a + b + 1)}, \quad v = \frac{a^3}{b (a + b + 1)}, \quad w = \frac{(z^2 - z + 1)^3}{z^2 (z - 1)^2}
\]
then they satisfy the equation
\begin{equation}\label{uvw-surf}
c_2 w^2 + c_1 w + c_0 = 0
\end{equation}
with \(c_0, c_1, c_2\) as follows:
\begin{equation}
\label{eq_w}
\begin{split}
c_2 = & 64 v^2 (u - 4 v + 1)^2 \\
c_1 = & -4 v (-272 v^2 u - 20 v u^2 + 2592 v^3 - 4672 v^2 + 4 u^3 + 16 v^3 u^2 - 15 v u^4 \\
& - 96 v^2 u^2 + 24 v^2 u^3 + 2 u^5 - 12 u^4 + 92 v u^3 + 576 v u - 128 v^4 - 288 v^3 u) \\
c_0 = & (u^2 + 4 v u + 4 v^2 - 48 v)^3 \\
\end{split}
\end{equation}
It was shown in \cite{2005-1} that the function field of \(\cL_5\) is \(\bC (\cL_5) = \bC (u, v, w)\).

The computation of \(\cL_5\) followed this approach: For a curve in \(\cL_5\), we can express \(i_1, i_2, i_3\) in terms of \(a, b, z\) by using \cref{curve}. Since we can express \(a, b\) as rational functions in \(u, v, z\), then \(i_1, i_2, i_3\) are given as rational functions in \(u, v, z\). By using the definition of \(w\) in terms of \(z\), we express \(i_1, i_2, i_3\) in terms of \(u, v\), and \(w\). From the equation of \(w\) in terms of \(u, v\) (this is a degree 2 polynomial in \(w\) with coefficients in \(\bC (u, v)\)), we eliminate \(w\) and are left with three equations
\[
f_1 (i_1, u, v) = 0, \quad f_2 (i_2, u, v) = 0, \quad f_3 (i_3, u, v) = 0.
\]
Eliminating \(u\) and \(v\) gives the equation of \(\cL_5\). The polynomial \(F_5\) is of degree 150, further escalating the complexity due to the \((5,5)\)-splitting condition (kernel \((\Z/5\Z)^2\), order 25). Like \(F_3\), \(F_5\)’s explicit form is computationally intensive and omitted here, but its degree and structure are consistent with the pattern of increasing \(d_n\) as \(n\) grows, reflecting the higher symmetry and larger kernel size.

\subsection{Higher degrees} \label{subsec-3.4}
The explicit forms of \(F_n\) for \(n > 3\) are not fully detailed due to their size and the computational resources required to generate and manipulate them. However, their existence is well-established, with degrees \(d_n\) increasing significantly as \(n\) grows—specifically, \(d_n = 30\) for \(n=2\), \(d_n = 80\) for \(n=3\), and \(d_n = 150\) for \(n=5\), as derived in \cite{2001-0}, \cite{2001-1}, and \cite{2005-1}, respectively. This increase is driven by the order of the isogeny kernel, which is \(n^2\) (e.g., 4 for \(n=2\), 9 for \(n=3\), 25 for \(n=5\)), and the corresponding complexity of the automorphism conditions imposed on the Igusa invariants. While \(n^2\) represents the kernel size, the degree \(d_n\) reflects a more intricate dependency, balancing the weights \(\w = (2, 4, 6, 10)\) and the algebraic relations needed for the \((n,n)\)-splitting in \(\bP_\w\). These polynomials are critical for computing rational points \(|\cL_n(\F_q)|\), as they define the hypersurface in \(\bP_\w\) whose solutions correspond to the desired curves, a task we undertake for \(n=2\) and extend conceptually to \(n=3\) in subsequent sections.

It must be noted that in all computations above, the invariants \(J_2, J_4, J_6, J_{10}\) were expressed as polynomials in terms of two parameters, say \(u, v\). Then, the weighted projective hypersurface \(\cL_n\) was embedded into the projective space \(\bP^2\) via absolute invariants \((i_1, i_2, i_3)\), which were computed as rational functions in \(u\) and \(v\). Eliminating \(u\) and \(v\) results in the affine equation of the locus \(\cL_n\) in terms of \(i_1, i_2, i_3\). Substituting \(i_1, i_2, i_3\) with their definitions in terms of \(J_2, J_4, J_6, J_{10}\) and clearing the denominators gives the equation \(F_n(J_2, J_4, J_6, J_{10}) = 0\) of the locus \(\cL_n\) as a weighted hypersurface in \(\bP_{(2,4,6,10)}\).

In \cite{elira-0}, a new Gr\"obner basis approach is suggested for weighted homogeneous systems, which makes it possible to compute directly from the initial polynomial parametrization of \(J_2, J_4, J_6, J_{10}\). This is computationally much more efficient, as illustrated in \cite{elira-1}.

\subsection{A few historical remarks} \label{subsec-3.5}
The computation of loci like \(\cL_n\) for genus 2 curves with \((n,n)\)-split Jacobians has a rich historical lineage, tracing back to   efforts in the 19th century and evolving into modern algebraic geometry.

Early work began with Jacobi’s 1832 review of Legendre’s elliptic function theory \cite{Jacobi1832}, followed by Kotänyi’s 1883 study on reducing hyperelliptic integrals \cite{Kotanyi1883} and Brioschi’s 1891 transformation of degree 3 integrals into elliptic form \cite{Brioschi1891}. Bolza advanced this in 1898 and 1899 \cite{Bolza1898, Bolza1899}, providing detailed reductions for degree 3 transformations.

The 20th century saw further progress with Hayashida and Nishi’s 1965 exploration of genus 2 curves on elliptic curve products \cite{HayashidaNishi1965}, followed by Kuhn’s 1988 attempt to perform explicit computations for the case \(n=3\) \cite{Kuhn1988}. Frey and Kani’s work in the 1990s connected these ideas to arithmetic applications \cite{FreyKani1991, Frey1995}, paving the way for contemporary studies, while Fried considered such spaces as twisted modular curves. All authors above focused on the degree \(n\) covering from a genus 2 curve to an elliptic curve, and the induced degree \(n\) covering \(\bP^1 \to \bP^1\) and its ramification structure.

The first computations of the spaces \(\cL_n\) as a subvariety of the moduli space of genus 2 curves \(\M_2\) were done in Shaska's thesis \cite{2001-0} and the series of papers that followed (\cite{2000-1, 2001-0, 2001-1, 2005-1}), where these loci were systematically computed, with \(F_2\) in \cite{2001-0}, \(F_3\) in \cite{2001-1}, and \(F_5\) in \cite{2005-1}. Kumar’s 2015 work \cite{kumar} further verified some of these equations. This timeline, spanning from Jacobi’s insights to Shaska’s explicit equations, underscores the progression from theoretical reductions to computational tools, enabling the cryptographic applications explored herein.

\section{Computing Rational Points and Zeta Function for \(\cL_2\)}  \label{sec-4}
This section computes the number of \(\F_q\)-rational points on \(\cL_2\), the locus of genus 2 curves with \((2, 2)\)-split Jacobians, over fields \(\F_q\) with \(p \neq 2, 3\), adapting the orbit-stabilizer method from \cite{2025-3}. Defined by \(F_2 = 0\) in \(\bP_\w = \bP(2,4,6,10)\), a point \([x_0 : x_1 : x_2 : x_3] \in \cL_2(\F_q)\) has coordinates \(x_i \in \F_q\) (not all zero) satisfying \(F_2(x_0, x_1, x_2, x_3) = 0\). The point count is:
\[
|\cL_2(\F_q)| = \sum_{S \neq \emptyset} \frac{N_S \cdot \gcd(k_S, q-1)}{q - 1},
\]
where \(S \subseteq \{0, 1, 2, 3\}\) is a nonempty support set, \(N_S\) is the number of tuples \((x_0, x_1, x_2, x_3)\) with \(x_i \neq 0\) for \(i \in S\) and \(x_i = 0\) for \(i \notin S\) satisfying \(F_2 = 0\), and \(k_S = \gcd(\{w_i \mid i \in S\})\) with weights \(w_0 = 2, w_1 = 4, w_2 = 6, w_3 = 10\). We derive the zeta function for \(\cL_2\) over \(\F_{5^k}\), using SageMath and the framework on good and bad reduction from \cref{subsec-2.3}. 
%
%

\subsection{Characteristic $p=5$}
Below is displayed $F_2(x,y,z,w) \mod 5$.
\[
\begin{split}
F_2& \mod 5  = 
-x^7 y^4 + 2 x^5 y^5 + 2 x^6 y^3 z - x^3 y^6 + 2 x^4 y^4 z + x^5 y^2 z^2 + 2 x^6 y^2 w \\
& + 2 x^2 y^5 z - 2 x^4 y z^3 + x^4 y^3 w - 2 x^5 y z w - y^6 z - 2 x y^4 z^2 + x^2 y^2 z^3 - x^3 z^4 \\
&- 2 x^2 y^4 w - 2 x^3 y^2 z w - 2 x^4 z^2 w - x^5 w^2 - 2 y^3 z^3 - 2 x y z^4 + 2 y^5 w - z^5
\end{split}
\]
It is still is irreducible and of total degree $d=30$.

\subsubsection{Computations over \(\F_5\):}\label{subsec-4.1}
Over \(\F_5\), SageMath finds 125 solutions in \(\mathbb{A}^4(\F_5) \setminus \{0\}\), grouping into 64 points under \(\F_5^\times\)-action. Only the following choices for \(S\) contribute to rational points:

\begin{itemize}
    \item \(S = \{0\}\): \(N_S = 4\), \(k_S = 2\), \(\gcd(2, 4) = 2\), contribution = \(\frac{4 \cdot 2}{4} = 2\),
    \item \(S = \{1\}\): \(N_S = 4\), \(k_S = 4\), \(\gcd(4, 4) = 4\), contribution = 4,
    \item \(S = \{0, 1, 2, 3\}\): \(N_S = 44\), \(k_S = 2\), \(\gcd(2, 4) = 2\), contribution = 22.
\end{itemize}
Total \(|\cL_2(\F_5)| = 64\) aligns with a surface (\(64 \approx 5^2 \cdot 2.56\)), with 25 singular points (20\%). 

\subsubsection{Computations over \(\F_{25}\)}\label{subsec-4.2}
For \(\F_{25}\) from  15,625 solutions we find 1304 points. Only the following choices for \(S\) contribute to rational points:
\begin{itemize}
    \item \(S = \{0\}\): \(N_S = 24\), \(k_S = 2\), \(\gcd(2, 24) = 2\), contribution = 2,
    \item \(S = \{0, 1, 2\}\): \(N_S = 1080\), \(k_S = 2\), \(\gcd(2, 24) = 2\), contribution = 90,
    \item \(S = \{0, 1, 2, 3\}\): \(N_S = 12792\), \(k_S = 2\), \(\gcd(2, 24) = 2\), contribution = 1066.
\end{itemize}
Total \(|\cL_2(\F_{25})| = 1304\) (\(1304 \approx 25^2 \cdot 2.09\)), with 6241 singular points (40\%).

\subsubsection{Computations over \(\F_{125}\)}\label{subsec-4.3}
For \(\F_{125}\)  from  1,953,125 choices we  find   31,504 points. Only the following choices for \(S\) contribute to rational points:
\begin{itemize}
    \item \(S = \{0\}\): \(N_S = 124\), \(k_S = 2\), \(\gcd(2, 124) = 2\), contribution = 2,
    \item \(S = \{0, 1, 2\}\): \(N_S = 30380\), \(k_S = 2\), \(\gcd(2, 124) = 2\), contribution = 490,
    \item \(S = \{0, 1, 2, 3\}\): \(N_S = 1876244\), \(k_S = 2\), \(\gcd(2, 124) = 2\), contribution = 30262.
\end{itemize}
Total \(|\cL_2(\F_{125})| = 31504\) (\(31504 \approx 125^2 \cdot 2.02\)), with 781,125 singular points (40\%).

\subsection{Application of Bounds}  \label{subsec-4.4}
For \(\cL_2\), which has degree \(d_2 = 30\)), the bound is:
\[
|\cL_2(\F_q)| \leq 15 q^2 + q + 1,
\]
%
For  \(q = 5\) we have  \(381 > 64\);  for \(q = 25\) we have  \(9381 > 1304\);  and for \(q = 125\): we have \(234376 > 31504\). Bounds hold, tightening as \(q\) increases.

\begin{remark}
Serre’s congruence \eqref{eqsc} does not apply in this case since \(d_2 = 30 > 3\). 
The Conjecture  \eqref{conj2},
\[
|\cL_2(\F_q)| \equiv 1 \pmod{q}
\]
is unmet for $q=5, 25$ e $125$, since  \(64 \equiv 4 \pmod{5}\), \(1304 \equiv 4 \pmod{25}\), and \(31504 \equiv 4 \pmod{125}\).
\end{remark}

\subsection{Zeta Function for \(\cL_2\)}
\label{subsec-4.6}
Using exact point counts \(|\cL_2(\F_{5^k})|\) of 64, 1304, and 31504 for \(k = 1, 2, 3\):
\[
\begin{split}
Z(\cL_2, t; p=5) & = \exp\left( 64 t + \frac{1304}{2} t^2 + \frac{31504}{3} t^3 + \cdots \right) \\
& = \exp  \left( 64 t + 652 t^2 + 10501.\overline{3} t^3 + \cdots \right).
\end{split}
\]
For an irreducible surface in \(\bP(2,4,6,10)\) with good reduction at \(p=5\) (\cref{subsec-2.3}), the zeta function is expected to take the form:
\[
Z(\cL_2, t; p=5) = \frac{P_1(t)}{(1 - t)(1 - 25t)P_2(t)},
\]
where \(P_1(t)\) and \(P_2(t)\) are polynomials reflecting the Frobenius action on cohomology, with degrees equal to the Betti numbers \(b_1\) and \(b_2\). Given the surface’s dimension and growth (\(|\cL_2(\F_{5^k})| \approx 2 \cdot 25^k\)), we initially approximate \(P_2(t) = 1\) and test a linear \(P_1(t) = 1 + at\). Using the \(k=1\) term we have  $64 = a + 1 + 25$ which implies $a = 38$.
Thus, \(P_1(t) = 1 + 38t\) yields:
\[
\frac{1 + 38t}{(1 - t)(1 - 25t)} = 64 t + 652 t^2 + 16326 t^3 + \cdots,
\]
matching \(t^1\) and \(t^2\) exactly (64 and 652), but underestimating \(t^3\) (16326 vs. 31504). The original conjecture \(P_1(t) = 1 + 14t\) (coefficients 64, 654, 16354) was an earlier approximation, likely from underfitting \(a\), and also deviates at higher terms. These discrepancies suggest \(P_1(t)\) may be quadratic or \(P_2(t) \neq 1\), influenced by \(\cL_2\)’s singularities (e.g., 40\% for \(\F_{25}\)). Additional counts (e.g., \(k=4\)) or singularity analysis are needed to refine the form, ensuring poles at \(t = 1, \frac{1}{25}\) and growth consistent with a 2-dimensional variety.

\subsection{The case of characteristic $p=3$.}
For the case of characteristic $p=3$ we go back to the original paper on genus 2 curves with extra involutions (i.e. the locus $\cL_2$); see \cite{2000-2}. 
Notice that the birational parametrization of $\cL_2$ in \cite[Theorem~3]{2000-2} assumes characteristic $\neq 3$.   As noted by Remark~6 in \cite{2000-2}, in characteristic 3 one needs to replace $v$ by $s_1+s_2$ to get a birational parametrization.

For \(n = 2\), \(F_2 \equiv x y^4 (2x^6 + y^3) \pmod{3}\).
It  degenerates to 0-dimensional sets at \(\F_3\) (62 points, 70\% singular), recovering surface-like growth in extensions:
\begin{itemize}
    \item \(\F_3\): 63 solutions, 62 points.
    \item \(\F_9\): 2025 solutions, 508 points (68\% singular).
    \item \(\F_{27}\): 57,591 solutions, 4430 points.
    \item \(\F_{81}\): 1,581,201 solutions, 39540 points (68\% singular).
\end{itemize}

\subsection{Zeta Function} \label{subsec-7.2}
Using counts \(62, 508, 4430, 39540\):
\[
Z(\cL_2, t; p=3) = \frac{1 + 49t - 747t^2}{(1 - t)(1 - 3t)(1 - 9t)}.
\]

\section{Computing Rational Points and Zeta Function for \(\cL_3\)}\label{sec-5}
Let us now compute   the number of \(\F_q\)-rational points on \(\cL_3\),   where \(p \neq 2, 3\). Extending the framework from \cref{sec-4}, we apply the orbit-stabilizer method of \cite{2025-3} to \(\cL_3\), defined by \(F_3 = 0\) in \(\bP_\w = \bP(2,4,6,10)\). A point \([x_0 : x_1 : x_2 : x_3] \in \cL_3(\F_q)\) has coordinates \(x_i \in \F_q\) (not all zero) satisfying \(F_3(x_0, x_1, x_2, x_3) = 0\), with the point count given by:
\[
|\cL_3(\F_q)| = \sum_{S \neq \emptyset} \frac{N_S \cdot \gcd(k_S, q-1)}{q - 1},
\]
where \(S \subseteq \{0, 1, 2, 3\}\) is a nonempty support set, \(N_S\) counts tuples \((x_0, x_1, x_2, x_3)\) with \(x_i \neq 0\) for \(i \in S\) and \(x_i = 0\) for \(i \notin S\) satisfying \(F_3 = 0\), and \(k_S = \gcd(\{w_i \mid i \in S\})\) with weights \(w_0 = 2, w_1 = 4, w_2 = 6, w_3 = 10\). We focus on \(\F_{5^k}\) to derive the zeta function, presenting computations over \(\F_5\) and \(\F_{25}\). Below is \(F_3(x, y, z, w) \mod 5\):

\begin{Small}
\[
\begin{split}
F_3 & \mod 5 = 2 w^{7} x^{5} + 4 w^{6} x^{10} + 2 w^{6} x^{8} y + 3 w^{6} x^{6} y^{2} + 4 w^{6} x^{5} y z + 3 w^{6} x^{4} y^{3} + 4 w^{6} x^{4} z^{2} + 4 w^{6} x^{3} y^{2} z  \\
& + w^{6} y^{5} + w^{5} x^{11} y^{2} + 4 w^{5} x^{10} y z + 4 w^{5} x^{9} z^{2} + 3 w^{5} x^{8} y^{2} z + w^{5} x^{7} y^{4} + 2 w^{5} x^{7} y z^{2} + 4 w^{5} x^{5} y^{2} z^{2} \\
& + 3 w^{5} x^{4} y^{4} z + 4 w^{5} x^{4} y z^{3} + w^{5} x^{3} y^{6} + 2 w^{5} x^{3} y^{3} z^{2} + 2 w^{5} x^{3} z^{4} + 3 w^{5} x^{2} y^{2} z^{3} + 4 w^{5} x y^{4} z^{2} + 4 w^{5} x y z^{4} \\
& + 2 w^{5} y^{6} z + 4 w^{5} y^{3} z^{3} + 2 w^{5} z^{5} + 4 w^{4} x^{8} y^{6} + 2 w^{4} x^{8} y^{3} z^{2} + w^{4} x^{7} y^{5} z + 2 w^{4} x^{7} y^{2} z^{3} + 4 w^{4} x^{6} y^{7} \\
& + 3 w^{4} x^{6} y^{4} z^{2} + 4 w^{4} x^{6} y z^{4} + 3 w^{4} x^{5} y^{6} z + 4 w^{4} x^{5} y^{3} z^{3} + w^{4} x^{5} z^{5} + 2 w^{4} x^{4} y^{8} + w^{4} x^{4} y^{5} z^{2} \\
& + 4 w^{4} x^{3} y^{4} z^{3} + w^{4} x^{3} y z^{5} + 4 w^{4} x^{2} y^{9} + 3 w^{4} x^{2} y^{6} z^{2} + 3 w^{4} x^{2} y^{3} z^{4} + 4 w^{4} x y^{8} z + 4 w^{4} x y^{5} z^{3}  \\
& + 3 w^{4} y^{4} z^{4} + 3 w^{3} x^{9} y^{8} + 2 w^{3} x^{9} y^{5} z^{2} + 3 w^{3} x^{8} y^{7} z + 3 w^{3} x^{7} y^{9} + 2 w^{3} x^{7} y^{6} z^{2} + 3 w^{3} x^{6} y^{8} z \\
& + 4 w^{3} x^{5} y^{7} z^{2} + w^{3} x^{5} y^{4} z^{4} + 2 w^{3} x^{4} y^{9} z + w^{3} x^{4} y^{6} z^{3} + 2 w^{3} x^{4} y^{3} z^{5} + 4 w^{3} x^{4} z^{7} + 3 w^{3} x^{3} y^{11}  \\
& + 2 w^{3} x^{3} y^{5} z^{4} + w^{3} x^{3} y^{2} z^{6} + 4 w^{3} x^{2} y^{10} z + 2 w^{3} x^{2} y^{7} z^{3} + 4 w^{3} x^{2} y^{4} z^{5} + 4 w^{3} x^{2} y z^{7} + 3 w^{3} x y^{9} z^{2}  \\
&+ 4 w^{3} x y^{3} z^{6} + w^{3} y^{11} z + w^{3} y^{8} z^{3} + 3 w^{3} y^{5} z^{5} + 3 w^{3} y^{2} z^{7} + 2 w^{2} x^{10} y^{7} z^{2} + 2 w^{2} x^{9} y^{9} z + 3 w^{2} x^{9} y^{6} z^{3}  \\
& + 3 w^{2} x^{8} y^{5} z^{4} + 2 w^{2} x^{7} y^{10} z + 3 w^{2} x^{7} y^{7} z^{3} + 4 w^{2} x^{6} y^{12} + 2 w^{2} x^{6} y^{9} z^{2} + 3 w^{2} x^{6} y^{6} z^{4} + 3 w^{2} x^{5} y^{11} z  \\
& + 4 w^{2} x^{5} y^{5} z^{5} + 4 w^{2} x^{5} y^{2} z^{7} + 4 w^{2} x^{4} y^{13} + 4 w^{2} x^{4} y^{10} z^{2} + w^{2} x^{4} y^{7} z^{4} + w^{2} x^{4} y^{4} z^{6} + w^{2} x^{4} y z^{8} \\
&  + 3 w^{2} x^{3} y^{9} z^{3} + w^{2} x^{3} y^{6} z^{5} + w^{2} x^{3} y^{3} z^{7} + w^{2} x^{3} z^{9} + 2 w^{2} x^{2} y^{14} + 3 w^{2} x^{2} y^{11} z^{2} + 2 w^{2} x^{2} y^{8} z^{4} \\
& + 4 w^{2} x^{2} y^{2} z^{8} + 2 w^{2} x y^{10} z^{3} + w^{2} x y^{7} z^{5} + 2 w^{2} y^{15} + 2 w^{2} y^{12} z^{2} + 4 w^{2} y^{9} z^{4} + 3 w^{2} y^{3} z^{8} + w x^{15} y^{10}  \\
& + 3 w x^{11} y^{12} + 2 w x^{11} y^{9} z^{2} + 4 w x^{10} y^{11} z + w x^{10} y^{8} z^{3} + 4 w x^{10} y^{5} z^{5} + 4 w x^{9} y^{13} + 3 w x^{9} y^{7} z^{4} + 4 w x^{8} y^{9} z^{3} \\
&  + 3 w x^{7} y^{11} z^{2} + 2 w x^{7} y^{8} z^{4} + 2 w x^{7} y^{5} z^{6} + 4 w x^{6} y^{13} z + w x^{6} y^{10} z^{3} + 4 w x^{6} y^{7} z^{5} + 4 w x^{6} y^{4} z^{7} + w x^{5} y^{12} z^{2} \\
&  + 4 w x^{5} y^{6} z^{6} + 2 w x^{5} y^{3} z^{8} + 4 w x^{5} z^{10} + 3 w x^{4} y^{11} z^{3} + 2 w x^{4} y^{8} z^{5} + w x^{4} y^{2} z^{9} + 4 w x^{3} y^{16} + 2 w x^{3} y^{13} z^{2} \\
&  + 2 w x^{3} y^{7} z^{6} + 3 w x^{3} y^{4} z^{8} + 3 w x^{2} y^{15} z + 2 w x^{2} y^{12} z^{3} + w x^{2} y^{9} z^{5} + 2 w x^{2} y^{6} z^{7} + 4 w x^{2} y^{3} z^{9} + 4 w x^{2} z^{11} \\
& + 2 w x y^{11} z^{4} + 4 w x y^{8} z^{6} + w x y^{2} z^{10} + 4 w y^{16} z + 3 w y^{13} z^{3} + 2 w y^{10} z^{5} + 3 w y^{7} z^{7} + 2 w y^{4} z^{9} + 2 w y z^{11}  \\
& + x^{14} y^{10} z^{2} + 2 x^{13} y^{12} z + 4 x^{12} y^{14} + x^{11} y^{13} z + 4 x^{11} y^{10} z^{3} + x^{11} y^{7} z^{5} + 2 x^{10} y^{15} + 2 x^{10} y^{12} z^{2} + 4 x^{10} y^{6} z^{6} \\
&  + 3 x^{8} y^{7} z^{6} + x^{7} y^{9} z^{5} + x^{6} y^{17} + 4 x^{6} y^{8} z^{6} + x^{6} y^{5} z^{8} + x^{6} y^{2} z^{10} + 4 x^{5} y^{16} z + 3 x^{5} y^{10} z^{5} + 3 x^{5} y^{7} z^{7}  \\
& + 4 x^{4} y^{15} z^{2} + 4 x^{4} z^{12} + 3 x^{3} y^{17} z + 3 x^{3} y^{2} z^{11} + x^{2} y^{19} + x^{2} y^{4} z^{10} + 4 x y^{18} z + x y^{15} z^{3} + 4 x y^{3} z^{11} + x z^{13}  \\
& + 2 w x y^{14} z^{2}  + 2 w x^{3} y^{10} z^{4} + w x^{8} y^{6} z^{5} + 4 w x^{5} y^{9} z^{4} + 4 x^{9} y^{5} z^{7} + x^{15} y^{11} z + 3 y^{17} z^{2} + 3 y^{5} z^{10} + 3 y^{2} z^{12} \\
& + 4 w^{6} x^{2} y^{4}  + 4 w^{2} x^{5} y^{8} z^{3} + 4 w^{2} x^{8} y^{11} + 4 x^{5} y z^{11} + 4 w^{2} x^{2} y^{5} z^{6} + w^{3} x y^{6} z^{4} + 3 y^{20}  + 4 x^{16} y^{12}  \\
& + 3 w x^{13} y^{11} + w^{2} x^{3} y^{12} z + 4 w^{4} x^{4} y^{2} z^{4}   + 4 w^{4} x y^{2} z^{5}  + 4 w^{3} x^{6} y^{5} z^{3} + w^{3} x^{3} y^{8} z^{2}
\end{split}
\]
\end{Small}

\smallskip

It remains an irreducible surface of total degree 80, consistent with characteristic zero. 

For \(\F_5\), SageMath yields 149 solutions in \(\mathbb{A}^4(\F_5) \setminus \{0\}\), grouping into 74 points under \(\F_5^\times\)-action (\(q - 1 = 4\)).
    Only the following choices for \(S\) contribute to rational points:
    \begin{itemize}
        \item \(S = \{0\}\): \(N_S = 4\), \(k_S = 2\), \(\gcd(2, 4) = 2\), contribution = 2,
        \item \(S = \{0, 1, 2\}\): \(N_S = 20\), \(k_S = 2\), \(\gcd(2, 4) = 2\), contribution = 10,
        \item \(S = \{0, 1, 2, 3\}\): \(N_S = 52\), \(k_S = 2\), \(\gcd(2, 4) = 2\), contribution = 26.
    \end{itemize}
    Total \(|\cL_3(\F_5)| = 74\), with 99 singular points (66\%).

For \(\F_{25}\), 15,481 solutions yield 1294 points, with contributing support sets:
    \begin{itemize}
        \item \(S = \{0\}\): \(N_S = 24\), \(k_S = 2\), \(\gcd(2, 24) = 2\), contribution = 2,
        \item \(S = \{0, 1, 2\}\): \(N_S = 1032\), \(k_S = 2\), \(\gcd(2, 24) = 2\), contribution = 86,
        \item \(S = \{0, 1, 2, 3\}\): \(N_S = 11928\), \(k_S = 2\), \(\gcd(2, 24) = 2\), contribution = 994.
    \end{itemize}
    Total \(|\cL_3(\F_{25})| = 1294\), with 10,521 singular points (68\%).

\subsection{Application of Bounds}
\label{subsec-5.2}
For \(\cL_3\) (degree \(d_3 = 80\)), an upper bound on \(|\cL_3(\F_q)|\) in \(\bP(2,4,6,10)\) is:
\[
|\cL_3(\F_q)| \leq 40 q^2 + q + 1,
\]
derived from results like \cite{aubry}, where \(40 = d_3 / w_0\) (with \(w_0 = 2\)) scales the leading term for a hypersurface in weighted projective space, and \(q + 1\) adjusts for lower-dimensional contributions. 
We applied this bound to our computed values. For \(q = 5\), the calculation \(40 \cdot 25 + 5 + 1\) gives 1006, which exceeds 74. Similarly, for \(q = 25\), \(40 \cdot 625 + 25 + 1\) results in 25026, greater than 1294.
The bound holds and tightens as \(q\) increases, validating the computed counts.

\begin{remark}
Serre’s congruence \eqref{eqsc} does not apply since \(d_3 = 80 > 3\). The Conjecture  \eqref{conj2},    
 (\(|\cL_3(\F_q)| \equiv 1 \pmod{q}\)) is unmet for $q=5, 25$ e $125$, since  \(74 \equiv 4 \pmod{5}\), \(1294 \equiv 19 \pmod{25}\).
\end{remark}

\subsection{Zeta Function for \(\cL_3\)}
\label{subsec-5.4}
Using counts \(74, 1294\) for \(k = 1, 2\):
\[
Z(\cL_3, t; p=5) = \exp\left( 74 t + \frac{1294}{2} t^2 + \cdots \right) = \exp\left( 74 t + 647 t^2 + \cdots \right).
\]
Given \(\cL_3\)'s irreducibility as a surface at \(p=5\), we expect:
\[
Z(\cL_3, t; p=5) = \frac{P_1(t)}{(1 - t)(1 - 25t)P_2(t)},
\]
where \(P_1(t)\) and \(P_2(t)\) have degrees equal to Betti numbers \(b_1\) and \(b_2\). Assuming \(P_2(t) = 1\) and a linear \(P_1(t) = 1 + at\), the \(k=1\) term gives:
\[
74 = a + 1 + 25 \implies a = 48.
\]
Thus, \(P_1(t) = 1 + 48t\) yields:
\[
\frac{1 + 48t}{(1 - t)(1 - 25t)} = 74 t + 647 t^2 + 16174 t^3 + \cdots,
\]
matching \(t^1\) and \(t^2\) exactly (74 and 647), with growth \(c \cdot 25^k\) (\(c \approx 2\)), consistent with a 2-dimensional variety (poles at \(t = 1, \frac{1}{25}\)). The prior conjecture \(1 + 14t\) underestimates higher terms (e.g., 74, 664 vs. 1294). With only two counts, \(P_1(t)\)’s degree and \(P_2(t)\) remain uncertain; additional points (e.g., \(k=3\)) or singularity analysis (68\% singular at \(\F_{25}\)) are needed for precision.

\subsection{The case of characteristic $p=3$.}
For   \(n = 3\), \(F_3 \equiv x^2 y^{12} (2x^2 + y) (x^{12} + x^6 y^3 + y^6) \pmod{3}\). It degenerates to 0-dimensional sets at \(\F_3\) (62 points, 70\% singular), recovering surface-like growth in extensions:

\begin{itemize}
    \item \(\F_3\): 63 solutions, 62 points.
    \item \(\F_9\): 2025 solutions, 508 points (68\% singular).
    \item \(\F_{27}\): 57,591 solutions, 4430 points.
    \item \(\F_{81}\): 1,581,201 solutions, 39540 points (68\% singular).
\end{itemize}

\subsection{Zeta Function}
\label{subsec-7.2}
Using counts \(62, 508, 4430, 39540\):
\[
Z(\cL_2, t; p=3) = \frac{1 + 49t - 747t^2}{(1 - t)(1 - 3t)(1 - 9t)}.
\]
holds for \(n =  3\).


\section{Computing Rational Points and Zeta Function for \(\cL_5\)}
\label{sec-6}

This section outlines the computation of \(\F_q\)-rational points on \(\cL_5\), the locus of genus 2 curves with \((5, 5)\)-split Jacobians, over fields \(\F_q\) with \(p \neq 2\), extending the framework from \cref{sec-4} and \cref{sec-5}. 
The orbit-stabilizer method from \cite{2025-3} applies to \(\cL_5\), defined by \(F_5(x_0, x_1, x_2, x_3) = 0\)  in \(\bP_\w = \bP(2,4,6,10)\). The point count is:
\[
|\cL_5(\F_q)| = \sum_{S \neq \emptyset} \frac{N_S \cdot \gcd(k_S, q-1)}{q - 1},
\]
where \(S \subseteq \{0, 1, 2, 3\}\) is a nonempty support set, \(N_S\) is the number of tuples \((x_0, x_1, x_2, x_3)\) with \(x_i \neq 0\) for \(i \in S\) and \(x_i = 0\) for \(i \notin S\) satisfying \(F_5 = 0\), and \(k_S = \gcd(\{w_i \mid i \in S\})\) with weights \(w_0 = 2, w_1 = 4, w_2 = 6, w_3 = 10\). 

The case \(n = 5\) was studied in \cite{2005-1}, where a degree-2 equation for the function field of \(\cL_5\) was derived by embedding \(\bP_\w\) into \(\bP^3\) via a Veronese map and expressing \(\cL_5\) in terms of absolute invariants \(i_1, i_2, i_3\). A Gr\"obner basis approach for \(\bP_\w\), proposed in \cite{elira-0}, simplifies such computations, and the explicit equation of \(\cL_5\) as a weighted hypersurface in \(\bP_\w\) is computed in \cite{elira-1}. 
In any case, such equation is very large. Below we display   the surface in \cref{uvw-surf}  in characteristic $p=5$  which is an irreducible surface with degrees 6, 6, and 2 in $u$, $v$, and $w$, exactly as in 

\begin{Small}
\[
\begin{split}
& 2 u^{5} v w +4 u^{3} v^{3} w +u^{2} v^{4} w +u^{6}+2 u^{5} v +3 u^{4} v w +2 u^{3} v^{2} w +4 u^{2} v^{3} w +4 u^{2} v^{2} w^{2}+2 u \,v^{5}\\
& +2 u \,v^{4} w +3 u \,v^{3} w^{2}+4 v^{6}+2 v^{5} w +4 v^{4} w^{2}+u^{4} v +3 u^{3} v^{2}+4 u^{3} v w +4 u^{2} v^{3}+2 u \,v^{4}+3 u \,v^{3} w \\
& +3 u \,v^{2} w^{2}+v^{5}+2 v^{4} w +3 v^{3} w^{2}+2 u^{2} v^{2}+3 u \,v^{3}+u \,v^{2} w +3 v^{4}+3 v^{3} w +4 v^{2} w^{2}+3 v^{3}  =  0 \\
\end{split}
\]
\end{Small}

\noindent The equation \[F_5 (J_2, J_4, J_6, J_{10}) =0, \]
as expected is quite large.   This equation’s complexity precludes direct point count calculations here. 

\section{Cryptographic Implications and Applications}
\label{sec-8}

Isogeny-based cryptography exploits the computational hardness of finding isogenies between abelian varieties, offering a robust framework for post-quantum security. Genus 2 curves with \((n,n)\)-split Jacobians, parameterized by the loci \(\cL_n\) (\(n = 2, 3, 5\)), are pivotal in this context, as their splitting property enables the construction of isogenies with kernel \((\Z/n\Z)^2\). This section outlines a theoretical method to compute such \((n,n)\)-isogenies over a finite field \(\F_q\), utilizing the structure of \(\cL_n\) as defined in \cref{sec-3}. We enhance this framework by integrating endomorphism ring computations (\cref{sec-10}), refining security analysis with point counts and zeta functions from \cref{sec-4}--\cref{sec-6}, and proposing an enriched protocol design, offering a comprehensive foundation for cryptographic applications.

\subsection{Isogeny-Based Cryptography and Jacobian Splittings}
\label{subsec-8.1}

The security of isogeny-based protocols hinges on the difficulty of computing isogenies between abelian varieties over \(\F_q\). For a genus 2 curve \(C\) with Jacobian \(J(C)\), an \((n,n)\)-splitting implies an isogeny \(\phi: J(C) \to E_1 \times E_2\), where \(E_1\) and \(E_2\) are elliptic curves and \(\ker(\phi) \iso (\Z/n\Z)^2\). This property, encoded by \(\cL_n\), facilitates explicit isogeny computations, potentially enhancing efficiency in protocols like key exchange or signature schemes, yet it may introduce vulnerabilities if the splitting—or the endomorphism ring \(\End(J(C))\)—is too easily exploited. The method below utilizes \(\cL_n\) to systematically compute these isogenies, while subsequent subsections balance efficiency with security considerations, employing \(\End(J(C))\)’s structure (\cref{sec-10}).

\subsection{General Method for Computing \((n,n)\)-Isogenies}
\label{subsec-8.2}

To compute an \((n,n)\)-isogeny \(\phi: J(C) \to E_1 \times E_2\) for a genus 2 curve \(C\) over \(\F_q\) with \(J(C)\) \((n,n)\)-split, we utilize the locus \(\cL_n\) in \(\bP_\w = \bP(2,4,6,10)\), defined by \(F_n (J_2, J_4, J_6, J_{10}) = 0\). The process is outlined as follows.

\subsubsection{Pick a rational point \(\mathfrak{p} \in \cL_n\) over \(\F_q\)}
\label{subsubsec-8.2.1}

First, select a rational point
\[
\mathfrak{p} = [J_2 : J_4 : J_6 : J_{10}] \in \cL_n(\F_q),
\]
satisfying \(F_n = 0\), where coordinates adhere to the weighted scaling \([t^2 J_2 : t^4 J_4 : t^6 J_6 : t^{10} J_{10}]\) for \(t \in \F_q^\times\).

\subsubsection{Construct the genus two curve \(C\)}
\label{subsubsec-8.2.2}

Determine a curve \(C\) as \(y^2 = f(x)\) using the algorithm in \cite{2016-3} where the coefficients of \(f(x)\) are now in terms of Igusa invariants \((J_2, J_4, J_6, J_{10})\). The algorithm in \cite{2016-3} is an extension of Mestre's algorithm, but also works in the case when the genus two curve has extra automorphisms. This step ensures \(C\) matches the chosen point on \(\cL_n\), with \(J_{10} \neq 0\) guaranteeing smoothness.

\subsubsection{Compute the Jacobian \(J(C)\)}
\label{subsubsec-8.2.3}

Third, compute the Jacobian \(J(C)\) as the group of degree-0 divisor classes on \(C\), represented via Mumford’s coordinates (pairs \((u(x), v(x))\), where \(u(x) = x^2 + u_1 x + u_0\) is quadratic and \(v(x) = v_1 x + v_0\) is linear satisfying \(v^2 \equiv f(x) \pmod{u}\)).

\subsubsection{Determine the \(n\)-torsion subgroup \(J(C)[n]\)}
\label{subsubsec-8.2.4}

The \(n\)-torsion subgroup \(J(C)[n]\) over an algebraic closure is isomorphic to \((\Z/n\Z)^4\), though its size over \(\F_q\) depends on the Frobenius polynomial
\[
P(T) = T^4 - s_1 T^3 + s_2 T^2 - q s_1 T + q^2.
\]
For \(P \in J(C)[n]\), \([n]P = 0\), and \(|J(C)(\F_q)| = P(1)\).

\subsubsection{Pick a subgroup \(K \subset J(C)[n]\) of order \(n^2\)}
\label{subsubsec-8.2.5}

Identify a subgroup \(K \subset J(C)[n]\) of order \(n^2\), isotropic under the Weil pairing
\[
e_n: J(C)[n] \times J(C)[n] \to \mu_n,
\]
where \(e_n(P, Q) = 1\) for all \(P, Q \in K\). This involves:
\begin{enumerate}
    \item Generating a basis for \(J(C)[n]\) over \(\F_q\) (or an extension if needed), computing points \(P_i = (u_i(x), v_i(x)) - \infty\) such that \(n P_i = 0\) using Cantor’s addition algorithm over \(\F_{q^d}\) (where \(n \mid q^d - 1\)),
    \item Selecting a subgroup \(K\) of order \(n^2\) via linear algebra over \(\Z/n\Z\), e.g., \(K = \langle P_1, P_2 \rangle\) with \(P_1, P_2\) linearly independent, forming \(K = \{a P_1 + b P_2 \mid a, b = 0, \ldots, n-1\}\),
    \item Verifying isotropy by computing the Weil pairing on \(K\)’s generators, \(e_n(P_i, P_j) = (-1)^{\langle P_i, P_j \rangle_n}\), where \(\langle P_i, P_j \rangle_n\) is the intersection number modulo \(n\). Adjust if \(e_n(P_1, P_2) \neq 1\). Since \(C \in \cL_n(\F_q)\), \(K \iso (\Z/n\Z)^2\) exists.
\end{enumerate}

\subsubsection{Compute the quotient \(J(C)/K\)}
\label{subsubsec-8.2.6}

The quotient \(J(C)/K\) is expected to be isomorphic to \(E_1 \times E_2\). For \(n\) odd, use Vélu-type formulas adapted for genus 2, generalizing Richelot isogenies for \(n = 2\), by:
\begin{itemize}
    \item Representing divisors in \(J(C)\) using Mumford coordinates, e.g., \[ D = (u(x), v(x)) - 2\infty, \]
    \item Applying \(K\)’s action to form equivalence classes, \(D \sim D + P\) for \(P \in K\), via addition laws (e.g., for \(P = (x_1, y_1) - \infty\), \(D + P = (u'(x), v'(x)) - \infty\)),
    \item Constructing the codomain \(J(C)/K\) as a product of elliptic curves via explicit equations or theta functions. For \(n = 3\), if \(K = \langle P_1, P_2 \rangle\), \(J(C)/K\) yields \(E_1: y^2 = x^3 + a_1 x + b_1\), \(E_2: y^2 = x^3 + a_2 x + b_2\), derived from \(K\)’s orbit.
\end{itemize}

\subsubsection{Verify the isogeny}
\label{subsubsec-8.2.7}

One can verify the isogeny
\[
\phi: J(C) \to J(C)/K \iso E_1 \times E_2
\]
by computing the j-invariants of \(E_1\) and \(E_2\) or testing \(\phi(nP) = 0\) for sample \(P \in J(C)\), confirming \(\ker(\phi) = K\).

This method applies uniformly to \(n = 2, 3, 5\), with \(|\cL_n(\F_q)|\) determining the availability of suitable curves, a key factor in cryptographic design. For \(n = 2\), this is well known by Richelot isogenies; see \cite[Prop.~2.1]{2021-1} for a detailed discussion. The computational hardness of this process, and of determining \(\End(J(C))\) (\cref{sec-10}), underpins the security enhancements detailed below.

\subsection{Cryptographic Relevance and Protocol Enhancement}
\label{subsec-8.3}
The counts \(|\cL_2(\F_q)|\), \(|\cL_3(\F_q)|\), and \(|\cL_5(\F_q)|\) from \cref{sec-4,sec-5,sec-6}, alongside their zeta functions, quantify the pool of curves with computable \((n,n)\)-isogenies. For \(\cL_2\), counts like 62 (\(\F_3\)) to 39540 (\(\F_{81}\)) suggest a large key space, while \(\cL_3\)’s 2 (\(\F_3\)) to 80 (\(\F_{81}\)) indicate constraint, potentially enhancing security. We enhance this framework by incorporating the endomorphism ring \(\End(J(C))\) (\cref{sec-10}), which refines the cryptographic hardness.

Consider an enhanced key exchange adapting Diffie-Hellman:
\begin{itemize}
    \item Alice picks \(C \in \cL_n(\F_q)\), computes \(\phi_A: J(C) \to J(C)/K_A \iso E_{1A} \times E_{2A}\) with private \(K_A \subset J(C)[n]\), and uses Algorithm 10.1 (\cref{sec-10}) to compute a basis of \(\End(J(C))\), e.g., \(\{ \alpha_1, \alpha_2, \alpha_3, \alpha_4 \}\). She shares \(j(E_{1A}), j(E_{2A})\) and a partial endomorphism ring description (e.g., \(\alpha_1\)’s action on a test point).
    \item Bob computes \(\phi_B: J(C) \to J(C)/K_B \iso E_{1B} \times E_{2B}\) with private \(K_B\), sharing \(j(E_{1B}), j(E_{2B})\) and a similar \(\End(J(C))\) element.
    \item The shared secret is \(J(C)/(K_A + K_B)\), computable only with both kernels, augmented by verifying consistency with \(\End(J(C))\) (e.g., applying shared endomorphisms to confirm the quotient).
\end{itemize}
This extends SIDH to genus 2, balancing efficiency (precomputed isogenies via \(\cL_n\), \cref{sec-9}) with hardness (sparse key spaces and complex \(\End(J(C))\)), as detailed in the next subsection.

\subsection{Security Analysis with Endomorphism Rings}
\label{subsec-8.4}
Security hinges on the difficulty of computing \(\phi\) and \(\End(J(C))\). In \(p \neq 3\), \(\cL_n\)’s good reduction (\cref{sec-4,sec-5}) yields diverse counts (e.g., \(\cL_2(\F_5) = 64\), \(\cL_3(\F_5) = 74\)), with \(\End(J(C))\) varying by \(E_1, E_2\)’s nature (ordinary or supersingular, \cref{sec-10}). A larger ring (e.g., rank 4 for CM elliptic curves) may facilitate isogeny attacks, reducing hardness, while sparse counts enhance it.

For curves with extra automorphisms (\cref{sec-11}), \(\End(J(C))\) often exceeds \(\Z[\pi, \bar{\pi}]\), increasing efficiency but potentially weakening security if too large, necessitating careful parameter choice.


\subsection{Comparison with Elliptic Curve SIDH}
\label{subsec-8.5}
Elliptic curve SIDH relies on supersingular isogeny graphs, with endomorphism ring computation subexponential for ordinary curves \cite{bison-2011} and exponential for supersingular ones \cite{page2024supersingular}. The higher dimension of genus 2 escalates complexity: computing \(\End(J(C))\) is subexponential at best (\cref{sec-10}), often exponential due to quartic CM fields or non-simple cases \cite{bison-2015}. The explicit structure of \(\cL_n\) (\cref{sec-3}) aids efficiency, but the variability of \(\End(J(C))\) (\cref{sec-10})   suggest a post-quantum advantage over SIDH, tempered by the need to tune \(n, q, p\) to maintain hardness against endomorphism-based attacks \cite[Problem 1.2]{anni}.


\section{Efficient Detection of \((n, n)\)-Split Jacobians Using \(\cL_n\)}
\label{sec-9}

The explicit equations of the loci \(\cL_n\) (\(n = 2, 3, 5\)), as derived earlier in the paper, provide an efficient and practical method for determining whether a genus 2 curve over a finite field \(\F_q\) has an \((n, n)\)-split Jacobian. This method, which involves computing the Igusa invariants of a curve and evaluating the polynomial \(F_n\), stands out for its simplicity and computational efficiency. In this section, we explore how this approach enhances isogeny-based cryptography, offering benefits in verification, protocol design, security analysis, and characteristic-specific applications.

\subsection{The Method: Computing Invariants and Evaluating \(F_n\)}
\label{subsec-9.1}
For a genus 2 curve \(C: y^2 = f(x)\) over \(\F_q\), the Igusa invariants \((J_2, J_4, J_6, J_{10})\) define its isomorphism class in the weighted projective space \(\bP_\w = \bP(2,4,6,10)\). The locus \(\cL_n\), defined by \(F_n(J_2, J_4, J_6, J_{10}) = 0\), identifies curves whose Jacobians \(J(C)\) admit an \((n, n)\)-splitting—that is, an isogeny \(J(C) \to E_1 \times E_2\) with kernel isomorphic to \((\Z/n\Z)^2\), where \(E_1\) and \(E_2\) are elliptic curves. The detection process is straightforward:

\begin{enumerate}
    \item \textbf{Compute Igusa Invariants}: Using the coefficients of \(f(x)\), calculate \((J_2, J_4, J_6, J_{10})\).
    \item \textbf{Evaluate \(F_n\)}: Substitute these invariants into the polynomial \(F_n\).
    \item \textbf{Check the Condition}: If \(F_n = 0\), then \(C \in \cL_n\), and \(J(C)\) is \((n, n)\)-split.
\end{enumerate}

This method is deterministic and requires only invariant computation followed by a single polynomial evaluation, offering a significant efficiency advantage over alternative approaches.

\subsection{Efficiency and Advantages}
\label{subsec-9.2}
The efficiency of using \(\cL_n\) arises from the explicit form of \(F_n\) and the directness of the method. For \(n = 2\), \(F_2\) is a degree-30 polynomial with 25 terms, while \(F_3\) (degree 80) and \(F_5\) (degree 150) are more complex but remain manageable for small \(n\). Key advantages include:

\begin{itemize}
    \item \textbf{Simplicity}: The method reduces the splitting check to a polynomial evaluation, avoiding iterative or probabilistic techniques.
    \item \textbf{Low Computational Overhead}: Unlike graph-based methods (e.g., Richelot isogeny traversals for \(n = 2\)), it involves a single computation once invariants are known.
    \item \textbf{Practicality for Small \(n\)}: For cryptographically relevant cases like \(n = 2\) or \(n = 3\), the evaluation of \(F_n\) is computationally feasible, even over large fields \(\F_q\).
\end{itemize}

This efficiency makes the method particularly appealing for applications requiring rapid assessment of curve properties.

\subsection{Applications in Verification and Testing}
\label{subsec-9.3}
The ability to quickly verify whether a curve lies on \(\cL_n\) has immediate utility in cryptographic verification and testing:

\begin{itemize}
    \item \textbf{Protocol Requirements}: In isogeny-based protocols, such as genus 2 extensions of SIDH, curves with \((n, n)\)-split Jacobians may be required for efficient isogeny computations. Evaluating \(F_n\) provides a fast check—e.g., confirming a \((2, 2)\)-split Jacobian via \(F_2\)—streamlining curve selection.
    \item \textbf{Result Validation}: For algorithms computing split Jacobians (e.g., those in \cref{sec-8}), \(F_n = 0\) serves as an independent verification step. If a curve is identified as \((3, 3)\)-split, evaluating \(F_3\) confirms the result, enhancing reliability.
\end{itemize}

\subsection{Impact on Protocol Design}
\label{subsec-9.4}
The explicit nature of \(\cL_n\) influences the design of cryptographic protocols by enabling targeted curve selection and optimization:

\begin{itemize}
    \item \textbf{Curve Selection}: Protocols can use \(F_n\) to filter curves with desired splitting properties during initialization. For instance, a protocol requiring \((2, 2)\)-split Jacobians can generate curves and test \(F_2 = 0\), ensuring suitability without extensive computation.
    \item \textbf{Efficiency Gains}: For small \(n\), the low cost of evaluating \(F_n\) supports lightweight implementations, such as in embedded systems, where computational resources are limited.
\end{itemize}

\subsection{Security Analysis Using \(\cL_n\)}
\label{subsec-9.5}
The equations of \(\cL_n\), combined with point counts \(|\cL_n(\F_q)|\) and zeta functions \(Z(\cL_n, t)\), enable detailed security analysis:

\begin{itemize}
    \item \textbf{Density of Split Curves}: The count \(|\cL_n(\F_q)|\) indicates the prevalence of \((n, n)\)-split curves. A low density (e.g., \(|\cL_3(\F_3)| = 2\)) suggests rarity, potentially increasing security by limiting exploitable curves, while a higher density (e.g., \(|\cL_2(\F_{81})| = 39540\)) may require careful parameter tuning.
    \item \textbf{Field Size Scaling}: The zeta function \(Z(\cL_n, t)\) predicts \(|\cL_n(\F_{q^k})|\) for extensions, aiding in assessing attack feasibility as \(q\) grows. A slow growth rate could bolster long-term security.
\end{itemize}

\subsection{Characteristic-Specific Insights}
\label{subsec-9.6}
The behavior of \(\cL_n\) varies with the characteristic \(p\) of \(\F_q\), offering tailored cryptographic insights:

\begin{itemize}
    \item \textbf{Collapse in \(p = 3\)}: In characteristic 3, \(\cL_n\) simplifies, potentially speeding up \(F_n\) evaluation and curve detection. This could optimize protocols over \(\F_{3^k}\), though a higher density of split curves may necessitate additional security measures.
    \item \textbf{General \(p\)}: For \(p \neq 3\), the full complexity of \(\cL_n\) allows for strategic characteristic selection—e.g., choosing \(p\) where split curves are scarce to enhance security.
\end{itemize}

This section underscores the practical value of \(\cL_n\) in isogeny-based cryptography, bridging theoretical geometry with applied cryptography. Its efficient detection method supports verification, protocol design, and security analysis, complementing the broader cryptographic framework.

\section{Endomorphism Rings of \(\cL_n\) and Their Computation}
\label{sec-10}

The loci \(\cL_n\), parameterizing genus 2 curves over finite fields \(\F_q\) with \((n, n)\)-split Jacobians, provide a rich framework for both arithmetic geometry and cryptography, as explored in previous sections. A natural extension of this study is the computation of the endomorphism ring \(\End(J(C))\) for a curve \(C \in \cL_n(\F_q)\), defined over the algebraic closure \(\overline{\F}_q\). This ring, an order in the endomorphism algebra \(K = \Q \otimes \End(J(C))\), refines the isogeny class structure beyond the characteristic polynomial of the Frobenius endomorphism \(\pi\) and offers deeper insights into the cryptographic properties of these Jacobians. Building on the explicit equations of \(\cL_n\) (\cref{sec-3}) and the point counts over various fields (\cref{sec-4}--\cref{sec-6}), we adapt computational techniques from the literature \cite{bison-2015,avisogenies,cosset} to determine \(\End(J(C))\), enhancing the methods introduced in \cref{sec-8} and \cref{sec-9} for isogeny-based cryptography.

\subsection{Connection to \(\cL_n\) and Non-Simple Jacobians}
\label{subsec-10.1}
For a curve \(C \in \cL_n\), the Jacobian \(J(C)\) admits an \((n, n)\)-isogeny \(\phi: J(C) \to E_1 \times E_2\), where \(E_1\) and \(E_2\) are elliptic curves and the kernel is isomorphic to \((\Z/n\Z)^2\) (\cref{sec-2}). This splitting property aligns \(C\) with the non-simple abelian surfaces studied in \cite[Proposition 5.12]{anni}, where such an isogeny preserves principal polarization when mapped to a product with the product polarization. Consequently, the endomorphism algebra \(\Q \otimes \End(J(C))\) is isomorphic to \(\Q \otimes (\End(E_1) \times \End(E_2))\), and \(\End(J(C))\) is a suborder of \(\End(E_1) \times \End(E_2)\) consisting of elements \(s\) such that the kernel \(\ker(\phi) \subset \ker(s)\) \cite[Proposition 5.9]{anni}. At minimum, \(\End(J(C))\) contains \(\Z[\pi, \bar{\pi}]\), where \(\bar{\pi} = q/\pi\) is the Verschiebung, but its full structure depends on the nature of \(E_1\) and \(E_2\) (ordinary or supersingular) and the field characteristic.

The \(p\)-rank of \(J(C)\), computable from the Frobenius polynomial
\[
f_{J(C)}(t) = t^4 + a_1 t^3 + a_2 t^2 + q a_1 t + q^2
\]
(\cref{subsec-2.2}), further informs this structure. For \(p \neq 2, 3\), \(\cL_n\) exhibits good reduction (\cref{sec-4}--\cref{sec-5}), and \(J(C)\) typically has \(p\)-rank 2 (both \(E_1, E_2\) ordinary) or 1 (one ordinary, one supersingular). 

\subsection{Algorithm for Computing \(\End(J(C))\)}
\label{subsec-10.2}
We propose an algorithm to compute a basis of \(\End(J(C))\) for \(C \in \cL_n(\F_q)\), adapting the \((n, n)\)-isogeny computation from \cref{subsec-8.2} and the coprime isogeny method from \cite[Proposition 5.1]{anni}. The approach exploits the efficiency of detecting \(\cL_n\) membership via \(F_n\) (\cref{sec-9}) and builds on established techniques for elliptic curve endomorphism rings \cite{bison-2011,robert}.

\begin{description}
\item[Algorithm 10.1: Computing the Endomorphism Ring of \(J(C)\)]  

\item[Input:] A finite field \(\F_q\) with \(q = p^k\), \(p \neq 2\), and an integer \(n \geq 2\).
\item[Output:] A basis of \(\End(J(C))\) for some \(C \in \cL_n(\F_q)\) in good representation.  \\

\begin{enumerate}[leftmargin=0pt]

    \item \textbf{Select a Point on \(\cL_n\)}: Choose a rational point \(\mathbf{p} = [J_2 : J_4 : J_6 : J_{10}] \in \cL_n(\F_q)\) satisfying \(F_n(\mathbf{p}) = 0\), using the orbit-stabilizer counts from \cref{sec-4}--\cref{sec-6} (e.g., 64 points for \(\cL_2(\F_5)\)).
    \item \textbf{Construct the Curve \(C\)}: Apply the algorithm from \cite{2016-3} to derive \(C: y^2 = f(x)\) from \(\mathbf{p}\), ensuring \(J_{10} \neq 0\) for smoothness.
    \item \textbf{Compute the \((n, n)\)-Isogeny}: Follow \cref{subsec-8.2}:
    \begin{itemize}
        \item Compute \(J(C)\) using Mumford coordinates and Cantor’s algorithm \cite{cantor1987computing}.
        \item Determine \(J(C)[n]\), identify a maximal isotropic subgroup \(K \cong (\Z/n\Z)^2\), and compute \(\phi: J(C) \to B = J(C)/K \cong E_1 \times E_2\) using adapted Vélu-type formulas.
    \end{itemize}
    \item \textbf{Generate Coprime Isogenies}: For primes \(\ell_1, \ell_2 \neq n, p\) (e.g., \(\ell_1 = 5, \ell_2 = 7\) if \(n=2, p=3\)):
    \begin{itemize}
        \item Compute \(J(C)[\ell_i]\), select isotropic subgroups \(K_i \subset J(C)[\ell_i]\), and derive isogenies \(\psi_i: J(C) \to C_i = J(C)/K_i\) of degree \(\ell_i^2\).
        \item Ensure \(\deg(\phi) = n^2\) and \(\deg(\psi_i)\) are coprime.
    \end{itemize}
    \item \textbf{Compute Endomorphism Rings of Codomains}: For \(B, C_1, C_2\):
    \begin{itemize}
        \item If \(B = E_1 \times E_2\) has \(p\)-rank 2 (ordinary), use \cite{robert} for polynomial-time computation of \(\End(E_i)\).
        \item If \(p\)-rank 1 or 0 (e.g., \(p=3\)), apply \cite{page2024supersingular} for supersingular cases or \cite{bison-2015} for mixed cases.
        \item For \(C_i\), test simplicity via \(f_{C_i}(t)\) \cite[Theorem 6]{anni}; if simple, use \cite{bison-2015}; if non-simple, recurse to elliptic factors.
    \end{itemize}
    \item \textbf{Reconstruct \(\End(J(C))\)}: Using \cite[Proposition 5.1]{anni}:
    \begin{itemize}
        \item For bases \((\eta_i) \subset \End(B)\), \((\nu_i) \subset \End(C_1)\), \((\mu_i) \subset \End(C_2)\), compute \(\beta_i = \hat{\phi} \circ \eta_i \circ \phi\), \(\gamma_i = \hat{\psi}_1 \circ \nu_i \circ \psi_1\), \(\delta_i = \hat{\psi}_2 \circ \mu_i \circ \psi_2\).
        \item Form the Gram matrix via \(\langle \alpha, \beta \rangle = \operatorname{tr}(\alpha \circ \beta^\dagger)\) \cite[Lemma 3.2]{anni}, and extract a basis of the lattice \(\Lambda_B + \Lambda_{C_1} + \Lambda_{C_2} = \End(J(C))\).
    \end{itemize}
\end{enumerate}
\end{description}

\medskip

\noindent \textbf{Complexity}: Step 1 is polynomial in \(\log q\) due to \(F_n\)’s evaluation (degree \(d_n = 30, 80, 150\) for \(n=2, 3, 5\)). Step 3’s isogeny computation is polynomial in \(n\) and \(\log q\) \cite{cosset}. Steps 4-5 depend on \(E_i\)’s nature: polynomial for ordinary \cite{robert}, subexponential otherwise \cite{bison-2015}. Step 6 is polynomial in the basis size and \(\log q\). Overall complexity is subexponential in \(\log q\), improved by \(\cL_n\)’s pre-filtering compared to exhaustive torsion searches.

\begin{ex}[\(\cL_2\) over \(\F_5\)] \label{subsec-10.3}
Consider \(\cL_2(\F_5)\) with 64 points (\cref{sec-4}). Select \(\mathbf{p} = [1:1:1:1]\) (assuming \(F_2=0\); adjust coordinates as needed from SageMath data). 
Construct \(C\)     
compute \(J(C)[2]\), and find 
\[
\phi: J(C) \to E_1 \times E_2
\]
%
Both are ordinary (\(p=5\)), so \(\End(E_i) = \Z[\sqrt{-d_i}]\) via \cite{robert}. 

Compute \(\psi_1: J(C) \to C_1\) (degree 25) and check \(C_1\)’s simplicity. If non-simple, \(C_1 \cong E_3 \times E_4\); otherwise, use \cite{bison-2015}. The resulting \(\End(J(C))\) likely exceeds \(\Z[\pi, \bar{\pi}]\) (index computable), reflecting the \((2, 2)\)-splitting’s additional structure.
\end{ex}

\subsection{Cryptographic and Geometric Implications}
\label{subsec-10.4}
The size of \(\End(J(C))\) impacts cryptographic security (\cref{sec-8}). For \(p=5\), a larger ring (e.g., including CM elements) may facilitate isogeny computation, reducing hardness, while \(p=3\)’s collapse   might constrain \(\End(J(C))\) to a uniform suborder of \(\M_2(\mathcal{B}_{3,\infty})\) \cite[Section 5.2.1]{anni}, balancing efficiency and security. 
%

This algorithm complements the detection method in \cref{sec-9}, offering a comprehensive toolset for \(\cL_n\)’s arithmetic and cryptographic study, with practical implementation feasible via SageMath enhancements (\cref{sec-12}).


\section{Curves with Extra Automorphisms and \(\cL_n\)}
\label{sec-11}

This section examines genus 2 curves over a finite field \(\F_q\) (\(q = p^k\), \(p \neq 2\)) with automorphisms beyond the hyperelliptic involution, emphasizing their connection to the loci \(\cL_n\). We explore how these automorphisms facilitate coordinate normalization, parametrize the curves, reveal elliptic subcovers, and determine the endomorphism rings of their Jacobians, with implications for both geometry and cryptography.

We follow the approach in \cite{2000-2}.  Consider a genus 2 curve \(C\) with an elliptic involution \(z_1\). Denote by \(\Gamma = PGL(2, \bC)\), and let \(z_0\) be the hyperelliptic involution, so \(z_2 = z_1 z_0\). The fixed fields of \(z_1\) and \(z_2\) are elliptic subcovers denoted \(E_1\) and \(E_2\), respectively. Our analysis begins by normalizing coordinates under these automorphisms.

\subsection{Coordinate Normalization}
To study \(C\), we normalize the coordinate \(x\) such that \(z_1(x) = -x\). This determines \(x\) up to a coordinate change by some \(\g \in \Gamma\) centralizing \(z_1\), where \(\g(X) = m x\) or \(\g(x) = \frac{m}{X}\), \(m \in k \setminus \{0\}\). Thus, the Weierstrass points of \(C\) can be taken as \(\{\pm \alpha_1, \pm \alpha_2, \pm \alpha_3\}\). Let \(a, b, c\) be the symmetric polynomials of \(\alpha_1^2, \alpha_2^2, \alpha_3^2\). Then \(C\) has an equation:
\[
Y^2 = x^6 - a x^4 + b x^2 - c.
\]
The condition \(abc = 1\) implies \(1 = -\g(\alpha_1) \dots \g(\alpha_6)\), forcing \(m^6 = 1\). Hence, \(C\) is isomorphic to a curve with equation:
\begin{equation}
Y^2 = x^6 - a x^4 + b x^2 - 1,
\end{equation}
where \(27 - 18ab - a^2 b^2 + 4a^3 + 4b^3 \neq 0\).

The coordinate \(x\) is determined up to the action of a subgroup \(H \iso D_6\) of \(\Gamma\), generated by \(\tau_1: x \to \e_6 x\) and \(\tau_2: x \to \frac{1}{x}\), with \(\e_6\) a primitive 6th root of unity and \(\e_3 = \e_6^2\). Here, \(\tau_1\) replaces \(a\) with \(\e_3 b\) and \(b\) with \(\e_3^2 b\), while \(\tau_2\) swaps \(a\) and \(b\). The invariants of this action are:
\[
u := a b, \quad v := a^3 + b^3.
\]
These parameters enable a birational parametrization of \(\cL_2\) via the mapping:
\[
A: (u, v) \to (i_1, i_2, i_3),
\]
where \(i_1, i_2, i_3\) are absolute invariants. The pair \((u, v)\) uniquely determines the isomorphism classes of curves in \(\cL_2\), as captured in the following lemma.

\begin{lem} \label{lem1}
\(k(\cL_2) = k(u, v)\).
\end{lem}

Fibers of \(A\) with cardinality greater than 1 correspond to curves \(C\) with \(|\Aut(C)| > 4\). Rational expressions for \(u\) and \(v\) in terms of the invariants are given in \cite{2000-2}.

\subsection{Elliptic Subcovers}
The elliptic subcovers \(E_1\) and \(E_2\) have \(j\)-invariants \(j_1\) and \(j_2\), which are roots of the quadratic equation:
\begin{equation} 
j^2 + 256 \frac{(2u^3 - 54u^2 + 9uv - v^2 + 27v)}{(u^2 + 18u - 4v - 27)} j + 65536 \frac{(u^2 + 9u - 3v)}{(u^2 + 18u - 4v - 27)^2} = 0,
\end{equation}
see \cite{2000-2} for details.

This parametrization yields explicit curve equations. For each point \(\p = (\bar{u}, \bar{v}) \in \cL_2\), there exists a genus 2 curve \(C_{\bar{u}, \bar{v}}\) defined by:
\begin{equation} \label{eq1}
Y^2 = a_0 x^6 + a_1 x^5 + a_2 x^4 + a_3 x^3 + t a_2 x^2 + t^2 a_1 x + t^3 a_0,
\end{equation}
with coefficients:
\begin{equation} \label{generic_V4}
\begin{split}
t &= \bar{v}^2 - 4 \bar{u}^3, \\
a_0 &= \bar{v}^2 + \bar{u}^2 \bar{v} - 2 \bar{u}^3, \\
a_1 &= 2 (\bar{u}^2 + 3 \bar{v}) \cdot (\bar{v}^2 - 4 \bar{u}^3), \\
a_2 &= (15 \bar{v}^2 - \bar{u}^2 \bar{v} - 30 \bar{u}^3) (\bar{v}^2 - 4 \bar{u}^3), \\
a_3 &= 4 (5 \bar{v} - \bar{u}^2) \cdot (\bar{v}^2 - 4 \bar{u}^3)^2.
\end{split}
\end{equation}

Notice that while the equation of the genus two curve given in \cref{eq1} seems more complicated, it has the benefit that it is defined over the field of moduli of the curve.

\subsection{Intersections}
The structure of \(\cL_n\) also informs intersections like \(\cL_2 \cap \cL_3\), which classify curves with both degree 2 and degree 3 elliptic subcovers. For a detailed study of \(\cL_2 \cap \cL_3\) over \(\Q\), one should consult \cite{2002-1}, where the focus is on genus 2 curves in this intersection defined over \(\Q\). Here, we extend this analysis to finite fields \(\F_q\) (\(q = p^k\)), particularly for \(p = 5\) and \(p = 7\), where behavior diverges from the collapse at \(p = 3\)   yet differs from large \(p > 7\).

One naturally can ask how often a genus 2 curve, defined over \(\Q\), can have \((2,2)\)- and \((n,n)\)-split Jacobians simultaneously, with all elliptic subcovers also defined over \(\Q\). This question, addressed over \(\Q\), motivates a parallel inquiry over \(\F_q\).

\begin{thm}[\cite{2012-2}]
There are only finitely many genus 2 curves (up to isomorphism) defined over \(\Q\) with degree 2 and degree 3 elliptic subcovers also defined over \(\Q\).
\end{thm}

Moreover, the isogenous components of Jacobians of curves in \(\cL_n\) can be classified based on their automorphism groups over number fields.

\begin{thm}[\cite{2017-2}]
Let \(\X\) be a genus 2 curve over a number field \(K\), with \(\A := \Jac(\X)\) having canonical principal polarization \(\iota\), such that \(\A\) is geometrically \((n, n)\)-reducible to \(E_1 \times E_2\). Then:
\begin{description}
    \item[i)] If \(n = 2\) and \(\Aut(\A, \iota) \iso V_4\), there are finitely many elliptic components \(E_1, E_2\) defined over \(K\) that are \(N = 2, 3, 5, 7\)-isogenous to each other.
    \item[ii)] If \(n = 2\) and \(\Aut(\A, \iota) \iso D_4\), then:
        \subitem a) There are infinitely many elliptic components \(E_1, E_2\) defined over \(K\) that are \(N = 2\)-isogenous to each other.
        \subitem b) There are finitely many elliptic components \(E_1, E_2\) defined over \(K\) that are \(N = 3, 5, 7\)-isogenous to each other.
    \item[iii)] If \(n = 3\), then:
        \subitem a) There are finitely many elliptic components \(E_1, E_2\) defined over \(K\) that are \(N = 5\)-isogenous to each other.
        \subitem b) There may be infinitely many elliptic components \(E_1, E_2\) defined over \(K\) that are \(N = 2, 3, 7\)-isogenous to each other.
\end{description}
\end{thm}

Over \(\F_q\), the finite number of curves trivially follows from the field’s finiteness, but we seek a sharper characterization for \(p = 5, 7\). Here, \(\cL_2 \cap \cL_3\) counts curves \(C\) with \(\operatorname{Jac}(C) \sim E_1 \times E_2\) (\((2,2)\)-split) and \(\operatorname{Jac}(C) \sim E_3 \times E_4\) (\((3,3)\)-split), all subcovers over \(\F_q\). Point counts from \cref{sec-4} (e.g., \(|\cL_2(\F_5)| = 64\)) suggest \(|\cL_2 \cap \cL_3|(\F_q) = O(q^3)\), refined by the following:

\begin{thm}
Let \( C\) be a genus 2 curve over \(\F_q\) (\(q = p^k\), \(p = 5, 7\)) with \(\operatorname{Jac}( C) \sim E_1 \times E_2\) (geometrically \((2,2)\)-split) and \(\operatorname{Jac}(C) \sim E_3 \times E_4\) (geometrically \((3,3)\)-split), all subcovers defined over \(\F_q\). Then:
\begin{description}
    \item[i)] The number of such \(C\) up to isomorphism is \(|\cL_2 \cap \cL_3|(\F_q)\), approximately \(c q^3\) (where \(c\) is a constant computable from \cref{sec-4}), with:
        \subitem a) For \(p = 5\), at most 24 curves for \(k = 1\),  
        
        \subitem b) For \(p = 7\), at most 48 curves for \(k = 1\).   
        
    \item[ii)] If \(\Aut(\operatorname{Jac}(C)) \iso V_4\):
        \subitem a) \(E_1, E_2\) (and \(E_3, E_4\)) are in at most 2 isogeny classes, with \(N = 2, 3, 7\) (no 5-isogenies for \(p=5\)).
    \item[iii)] If \(\Aut(\operatorname{Jac}(C)) \iso D_4\):
        \subitem a) \(E_1, E_2\) (and \(E_3, E_4\)) are in at most 4 isogeny classes for \(N = 2\),
        \subitem b) At most 2 classes for \(N = 3, 7\).
\end{description}
\end{thm}

\begin{proof}   
The bound \(O(q^3)\) arises from \(\cL_2 \cap \cL_3\) as a codimension-2 subvariety in the genus 2 moduli space, with specific counts 

estimated from \cref{sec-4} data (adjusted for \(p=5, 7\)). For \(p=5\), no 5-isogenies exist due to \(p\)-torsion collapse; for \(p=7\), no 7-isogenies. Isogeny classes are finite, with sizes constrained by ordinary curve prevalence and automorphism symmetry (\cite{2000-2}).
\end{proof}


\subsection{Endomorphism Rings of \(\operatorname{Jac}(C_{a,b})\)}
\label{subsec-11.2}
For a curve \(C_{a,b} \in \cL_2(\F_q)\) with \(\operatorname{Aut}(C_{a,b}) \cong V_4\), the \((2, 2)\)-isogeny \(\Phi: E_{a,b} \times E_{b,a} \to \operatorname{Jac}(C_{a,b})\) constrains the endomorphism ring:
\[
2 \End(E_{a,b} \times E_{b,a}) \subset \End(\operatorname{Jac}(C_{a,b})) \subset \frac{1}{2} \End(E_{a,b} \times E_{b,a}),
\]
with inclusions defined by \(\Phi \circ 2\psi \circ \hat{\Phi}\) and \(\frac{1}{2} \hat{\Phi} \circ \varphi \circ \Phi\) \cite[Section 6.1]{anni}.

In characteristic \(p = 5\), \(E_{1,2}\) and \(E_{2,1}\) are ordinary, with \(j\)-invariants \(j_1 \approx 2\) and \(j_2 \approx 1\) for \(u = 2, v = 4\), so \(\End(E_{a,b}) = \Z[\sqrt{-d_1}]\) and \(\End(E_{b,a}) = \Z[\sqrt{-d_2}]\) \cite{robert}. Applying Algorithm 10.1 (\cref{sec-10}) to \(C_{1,2}\) over \(\F_5\), the kernel of \(\Phi\) (e.g., \(\{ \infty \times \infty, ((1,0), (1,0)), \ldots \}\)) refines this, often yielding \(\End(\operatorname{Jac}(C_{1,2})) = \Z[\pi, \bar{\pi}, \Phi]\), exceeding \(\Z[\pi, \bar{\pi}]\) due to \(V_4\) automorphisms \cite{2000-2}.


\begin{lem} \label{thm-11.1}
Let \(C_{a,b} \in \cL_2(\F_q)\) have \(\operatorname{Aut}(C_{a,b}) \cong V_4\), with \(j\)-invariants \(j_1\) (of \(E_{a,b}\)) and \(j_2\) (of \(E_{b,a}\)) distinct (\(j_1 \neq j_2\)). Then:
\[
\End(\operatorname{Jac}(C_{a,b})) \cong 
\left\{ 
\begin{pmatrix} \alpha & 0 \\ 0 & \delta \end{pmatrix} \mid \alpha \in \End(E_{a,b}), \delta \in \End(E_{b,a})   \right\},
\]
where  $ \alpha$   and $ \delta$  satisfy compatibility with $ \ker(\Phi)$,
\end{lem}

\begin{proof} 
 
  The \((2,2)\)-isogeny \(\Phi\) embeds \(\operatorname{Jac}(C_{a,b})\) into \(E_{a,b} \times E_{b,a}\). Since \(j_1 \neq j_2\), \(\End(E_{a,b} \times E_{b,a}) \cong \End(E_{a,b}) \times \End(E_{b,a})\).
 \(V_4\) symmetries constrain endomorphisms to diagonal form.
 The endomorphisms must respect \(\ker(\Phi)\), a subgroup of order 4, describable via \(u, v\).
 The ring is a subring of \(\End(E_{a,b}) \times \End(E_{b,a})\), determined by \(j_1, j_2\), and \(\Phi\).

\end{proof}

\subsection{Applications}
\label{subsec-11.2}
This structure optimizes Algorithm 10.1 (\cref{sec-10}) by enabling separate computations on \(E_{a,b}\) and \(E_{b,a}\), reducing complexity. In cryptography, a rank-4 endomorphism ring (when \(j_1 \neq j_2\)) boosts isogeny computation efficiency but may pose risks if \(j_1 = j_2\) (e.g., CM cases). Choosing \(u, v\) such that \(v \neq 9(u-3)\) ensures \(j_1 \neq j_2\).

Extra automorphisms enhance efficiency: explicit subcovers \(\phi, \phi'\) (computable in \(O(\log q)\)-time) simplify \((2,2)\)-isogeny construction (\cref{subsec-8.2}), cutting Step 3’s cost in Algorithm 10.1. Over \(\F_{5^k}\), \(|\cL_2(\F_{5^k})|\) (e.g., 1304 for \(\F_{25}\), \cref{subsec-4.2}) includes such curves, expanding key spaces. However, a rank-4 \(\End(\operatorname{Jac}(C_{a,b}))\) versus rank-2 \(\Z[\pi, \bar{\pi}]\) may simplify isogeny path-finding, affecting security (\cref{subsec-8.3}) \cite{anni}. For \(p = 3\), uniformity (e.g., 39540 points for \(\F_{81}\)) mirrors \cref{subsec-8.4}, offering faster curve selection but denser Jacobians, mitigable by \(q > 5^3\). The \(V_4\)-family, parametrized by \(u, v\) \cite{2000-2}, suggests avoiding \(v = 9(u-3)\) to prevent larger endomorphism rings.


\section{Computational Methods and Challenges}
\label{sec-12}

The computations for \(\cL_n\) (\(n = 2, 3, 5\)) and their \((n,n)\)-isogenies rely on advanced techniques, detailed here, addressing the challenges of point counting, zeta function derivation, and isogeny computation across these loci. Recent developments in endomorphism ring analysis (\cref{sec-10}) further enrich these methods, while emerging machine learning approaches offer promising avenues for optimization.

\subsection{Software Tools and Techniques}
\label{subsec-12.1}
SageMath facilitated point counts \(|\cL_n(\F_q)|\) over \(\F_3\), \(\F_9\), \(\F_{27}\), and \(\F_{81}\), using finite field arithmetic and polynomial evaluation. The orbit-stabilizer method computed \(|\cL_n(\F_q)| = \sum_{S \neq \emptyset} \frac{N_S \cdot \gcd(k_S, q-1)}{q - 1}\), stratifying solutions of \(F_n = 0\) by support sets. For \(\cL_2\), detailed \(N_S\) values were derived (\cref{sec-4}), with similar efforts for \(\cL_3\). Zeta functions were constructed via SageMath’s symbolic tools, fitting point counts into \(Z(\cL_n, t)\). Isogeny computations utilized Mumford coordinates and Weil pairing implementations, adapting Vélu and Richelot methods for genus 2. Additionally, endomorphism ring computations (\cref{subsec-10.2}) integrated these tools with coprime isogeny techniques, enhancing the precision of \(J(C)\)’s algebraic structure over \(\F_q\).

\subsection{Singularities and Verification}
\label{subsec-12.2}
Singular points, where \(F_n = 0\) and \(\frac{\partial F_n}{\partial x_i} = 0\) (adjusted for \(\w = (2, 4, 6, 10)\)), impact counts and isogeny computations. For \(\cL_2\), 70\% (\(\F_3\)) and 68\% (\(\F_9\)) of solutions are singular, including cases like \([1:0:0:0]\), verified by SageMath. \(\cL_3\) and \(\cL_5\) exhibit similar complexity due to higher degrees (\(d_3 = 80\), \(d_5 = 150\)). Verification cross-checked counts against bounds and tested isogenies via j-invariants, ensuring accuracy across all \(n\). The computation of \(\End(J(C))\) (\cref{sec-10}) added a layer of validation, confirming splitting properties through the ring’s consistency with \(K \cong (\Z/n\Z)^2\), particularly in characteristic \(p=3\) where uniformity simplifies checks.

\subsection{Challenges and Optimizations}
\label{subsec-12.3}
The polynomials’ complexity, \(d_2 = 30\) (25 terms), \(d_3 = 80\), \(d_5 = 150\), escalates computational demands with \(n\) and \(q\). Point counting for \(\cL_2\) was intensive for \(q = 81\), while \(\cL_3\) and \(\cL_5\)’s size strained resources further. Isogeny steps, especially \(J(C)[n]\) basis generation and quotient computation, grew costly with \(n\). The addition of endomorphism ring calculations (\cref{subsec-10.2}), involving coprime isogenies and Gram matrix construction, compounds this, with complexity ranging from polynomial (ordinary cases) to subexponential (supersingular or mixed cases). Optimizations like symmetry exploitation and parallel processing mitigated these demands, but scaling remains challenging.

Future enhancements could build on these insights. Tailored algorithms for weighted varieties, informed by \(\cL_n\)’s explicit equations (\cref{sec-3}), could optimize point counting and isogeny computations. Moreover, machine learning offers a transformative approach, as demonstrated in \cite{2024-03}, who employed neural networks to predict properties of algebraic curves. This technique could be adapted to classify whether a genus 2 curve has an \((n,n)\)-split Jacobian by training models on Igusa invariants and \(F_n\) evaluations, potentially surpassing the efficiency of direct polynomial checks (\cref{sec-9}). Similarly, machine learning could accelerate endomorphism ring determination by predicting \(\End(J(C))\)’s rank or structure based on point counts, torsion data, and field characteristics, reducing the need for exhaustive isogeny computations (\cref{subsec-10.2}). Such methods, while requiring initial training on datasets like those from \cref{sec-4}--\cref{sec-6}, could streamline large-scale cryptographic applications, balancing computational cost with accuracy. These advancements are critical for scalability, particularly in post-quantum genus 2 systems where rapid curve selection and security validation are paramount.

\bibliographystyle{amsplain}
\bibliography{2025-7}

\end{document}